\documentclass[10pt]{article}
\usepackage[top=25mm,bottom=25mm,left=25mm,right=25mm]{geometry}

\usepackage{amsmath,amssymb,amsthm,latexsym,amscd,stmaryrd,mathrsfs,longtable,color,url}

\newtheoremstyle{standard}%
{9pt}%
{9pt}%
{\it}
{}%
{\bfseries}%
{}
{ }%
{#3}%

\newcommand{\db}[1]{(\!({#1})\!)}

\newcommand{\centralc}{C}
\newcommand{\subL}{P}

\numberwithin{equation}{section}

\newcommand{\N}{{\mathbb N}}
\newcommand{\Z}{{\mathbb Z}}
\newcommand{\Q}{{\mathbb Q}}

\newcommand{\C}{{\mathbb C}}

\newcommand{\wi}{i}

\newcommand{\wh}{h}
\newcommand{\Har}{H}

\newcommand{\wn}{n}

\newcommand{\wx}{x}

\newcommand{\wl}{l}

\newcommand{\rankL}{d}
\newcommand{\mN}{N}
\newcommand{\mW}{W}

\newcommand{\module}{M}

\newcommand{\mn}{p}
\newcommand{\hei}{{\mathfrak h}}

\newcommand{\sU}{{\mathscr U}}

\newcommand{\nor}{\begin{subarray}{c}\circ\\\circ\end{subarray}}

\newcommand{\fg}{{\mathfrak g}}
\newcommand{\fh}{{\mathfrak h}}

\newcommand{\ul}[1]{{#1}}
\newcommand{\lu}{u}

\newcommand{\lv}{v}

\newcommand{\mK}{K}
\newcommand{\lom}{r}
\newcommand{\lJ}{s}
\newcommand{\lE}{t}
\newcommand{\ExB}{E}

\newcommand{\lao}{r}
\newcommand{\laJ}{s}
\newcommand{\laE}{t}

\newcommand{\sv}{P}
\newcommand{\vac}{{\mathbf 1}}
\newcommand{\lattice}{L}

\newcommand{\lvE}{\ExB_{t}\lv}

\DeclareMathOperator{\id}{id}

\DeclareMathOperator{\Span}{Span} 
\DeclareMathOperator{\wt}{wt}

\DeclareMathOperator{\rank}{rank} 

\DeclareMathOperator{\tw}{tw}

\newtheorem{lemma}{Lemma}[section]
\newtheorem{theorem}[lemma]{Theorem}
\newtheorem{proposition}[lemma]{Proposition}

\theoremstyle{definition}
\newtheorem{definition}[lemma]{Definition}

\newtheorem{remark}[lemma]{Remark}

\theoremstyle{standard}

\title{The irreducible weak modules for the fixed point subalgebra of the vertex algebra associated to
a non-degenerate even lattice by an automorphism of order $2$ (Part $1$)}
\author{Kenichiro Tanabe\footnote{Research was partially supported by the Grant-in-aid
(No. 18K03198) for Scientific Research, JSPS.}\\\\
Department of Mathematics\\
Hokkaido University\\
Kita 10, Nishi 8, Kita-Ku, Sapporo, Hokkaido, 060-0810\\
Japan\\\\
ktanabe@math.sci.hokudai.ac.jp}
\date{}
\begin{document}
\maketitle

\begin{abstract}
Let $V_{\lattice}$ be the vertex algebra  associated to a non-degenerate even lattice $\lattice$,
$\theta$ the automorphism of $V_{\lattice}$ induced from the $-1$-isometry of $\lattice$, and
$V_{\lattice}^{+}$ the fixed point subalgebra of $V_{\lattice}$ under the action of $\theta$.
In this series of papers, we classify the irreducible weak $V_{\lattice}^{+}$-modules
and show that any irreducible weak $V_{\lattice}^{+}$-module 
is isomorphic to a weak submodule of some irreducible weak $V_{\lattice}$-module or 
to a submodule of some irreducible $\theta$-twisted $V_{\lattice}$-module.

In this paper (Part 1), we show that when the rank of $L$ is $1$,
every non-zero weak $V_{\lattice}^{+}$-module contains a non-zero $M(1)^{+}$-module,
where $M(1)^{+}$ is the fixed point subalgebra of the Heisenberg vertex operator algebra $M(1)$
under the action of $\theta$.
\end{abstract}



\tableofcontents
\section{\label{section:introduction}Introduction}
Let $V$ be a vertex algebra, $G$ a finite automorphism group of $V$,
and $V^G$ the fixed point subalgebra of $V$ under the action of $G$: $V^{G}=\{u\in V\ |\ gu=u\mbox{ for all }g\in G\}$.
The  fixed point subalgebras play an important role in the study of vertex algebras, particularly in the 
construction of interesting examples.
For example,  let $V_{\lattice}$ be the vertex algebra associated to a non-degenerate even lattice $\lattice$ of finite rank
and $\theta$ the automorphism of $V_{\lattice}$ of order $2$ induced from the $-1$-isometry of $\lattice$.
We write $V_{\lattice}^{\pm}=\{a \in V_{\lattice}\ |\ \theta(a)=\pm a\}$ for simplicity.
Then, the moonshine vertex algebra 
$V^{\natural}$ is constructed from $V_{\Lambda}^{+}$ and an irreducible $V_{\Lambda}^{+}$-module
in \cite{FLM} where $\Lambda$ is the Leech lattice.
We remark that $V_{\Lambda}^{+}$ is also a fixed point subalgebra of $V^{\natural}$
under the action of an automorphism of order $2$.

One of the main problems about $V^G$ is to describe the $V^G$-modules in terms of $V$ and $G$.
If $V$ is a vertex operator algebra, it is conjectured that under some conditions on $V$,
every irreducible $V^G$-module is a submodule of some irreducible $g$-twisted 
$V$-module for some $g\in G$ (cf. \cite{DVVV1989}).
The conjecture is confirmed for many examples including $V_{\lattice}^{+}$
where $L$ is a positive definite even lattice (cf. \cite{AD2004,DLaTYY2004,DN1999-1,DN1999-2,DN2001,TY2007,TY2013}).
In those examples, they classify the irreducible $V^{G}$-modules directly
by investigating the Zhu algebra, which is an associative $\C$-algebra introduced in \cite{Z1996},
since \cite[Theorem 2.2.1]{Z1996} says that
for a vertex operator algebra $V$
there is a one to one correspondence between the set of all isomorphism classes of irreducible $\N$-graded weak $V$-modules 
and that of irreducible modules for the Zhu algebra associated to $V$, where we note that
an arbitrary $V$-module is automatically an $\N$-graded weak $V$-module.
We recall some results on the representations of $V_{L}$ and $V_{L}^{+}$.
It is well known that the vertex algebra $V_{L}$ is a vertex operator algebra 
if and only if $\lattice$ is positive definite.
It is shown in \cite[Theorem 3.1]{Dong1993}
that
$\{V_{\lambda+\lattice}\ |\ \lambda+\lattice\in \lattice^{\perp}/L\}$ is a complete set of representatives of equivalence classes of
the irreducible weak $V_{\lattice}$-modules (see Definition \ref{definition:weak-module} for the definition of a weak module),
where $L^{\perp}$ is the dual lattice of $\lattice$. 
It is also shown in \cite[Theorem 3.16]{DLM1997} that
every weak $V_{\lattice}$-module is 
completely reducible.
The corresponding results for $\theta$-twisted weak $V_{\lattice}$-modules are obtained in \cite{Dong1994}.
When the lattice $\lattice$ is positive definite, 
the irreducible $V_{\lattice}^{+}$-modules 
are classified in \cite{AD2004,DN1999-2},
the fusion rules are determined in \cite{Abe2001, ADL2005}
and it is established that $V_{\lattice}^{+}$ is
$C_2$-cofinite in \cite{ABD2004, Yamsk2004} and rational in \cite{Abe2005, DJL2012}.
Thus, in this case it follows from \cite[Theorem 4.5]{ABD2004} that
$V_{\lattice}^{+}$ is regular and every irreducible weak $V_{\lattice}^{+}$-module
is an irreducible $V_{\lattice}^{+}$-module.
Here, it is worth mentioning that if a vertex operator algebra $V$
is simple and $C_2$-cofinite, then so is $V^G$ for any
finite solvable automorphism group $G$ of $V$ by \cite{Mi2015}.

We consider the case that $\lattice$ is not positive definite.
In this case, vertex algebras $V_{\lattice}$ and $V_{\lattice}^{+}$ are also related to $V^{\natural}$. 
In fact, in \cite{B1992}, Borcherds  constructs the monster Lie algebra 
from the tensor product $V^{\natural}\otimes_{\C}V_{II_{1,1}}$ 
where $II_{1,1}$ is the even unimodular Lorentzian lattice of rank $2$
and uses this Lie algebra to prove the moonshine conjecture stated in \cite{CN1979}.
Moreover, let $II_{25,1}$ be the even unimodular Lorentzian lattice of signature $(25,1)$
and $\Gamma$ a Niemeier lattice, where we recall that the Leech lattice $\Lambda$ is one of the Niemeier lattices.
It is known that the direct sum $\Gamma\oplus II_{1,1}$
is isomorphic to $II_{25,1}$, which implies that
 $V_{II_{25,1}}\cong V_{\Gamma}\otimes V_{II_{1,1}}$
and hence  $V_{II_{25,1}}^{+}\cong V_{\Gamma}^{+}\otimes V_{II_{1,1}}^{+}\oplus
V_{\Gamma}^{-}\otimes V_{II_{1,1}}^{-}$.
Thus,  $V_{\Gamma}^{+}$'s are related to each other through $V_{II_{25,1}}^{+}$.
Since $V_{\Lambda}^{+}$ is a fixed point subalgebra of $V^{\natural}$ as mentioned above
and there are several interesting connections between $II_{25,1}$ and 
 the Niemeier lattices obtained in 
\cite{B1985} and \cite{CS1982-1,CS1982-2} (see also \cite[Chapters 24 and 26]{CS1999}),
one may expect that 
the study of $V_{II_{25,1}}$ and its subalgebras including $V_{II_{25,1}}^{+}$ provides a better understanding of the moonshine vertex algebra $V^{\natural}$ and related algebras.
This is one of the motivations to study representations of $V_{\lattice}^{+}$.
As mentioned above $V_{\lattice}$ is not a vertex operator algebra in this case, however,
we note that the conjecture about representations for $V^G$ above makes sense even for (weak) modules for a vertex algebra $V$.
For $V_{\lattice}^{+}$-modules, 
the irreducible $V_{\lattice}^{+}$-modules 
are classified by using the Zhu algebra in \cite{J2006, Yamsk2008} and it is established that $V_{\lattice}^{+}$ is $C_2$-cofinite in
\cite{JY2010} and rational in \cite{Yamsk2009}.
However, the study of weak $V_{\lattice}^{+}$-modules has not progressed
in spite of the fact that $V_{\lattice}^{+}$ itself is a weak module
but not a module for $V_{\lattice}^{+}$,
because of the absence of useful tool like the Zhu algebras for weak modules.

The following is the main result of this series of papers, which implies that
for any non-degenerate even lattice $\lattice$ of finite rank,
every irreducible weak $V_{\lattice}^{+}$-module
is isomorphic to a weak submodule of some irreducible weak $V_{\lattice}$-module or to a submodule of some 
 irreducible $\theta$-twisted $V_{\lattice}$-module.
Namely, the conjecture above holds for irreducible weak $V_{\lattice}^{+}$-modules.

\begin{theorem}
\label{theorem:classification-weak-module}
Let $\lattice$ be a non-degenerate even lattice of finite rank with a bilinear form $\langle\ ,\ \rangle$.
The following is a complete set of representatives of equivalence classes of the irreducible weak $V_{\lattice}^{+}$-modules:
\begin{enumerate}
\item
$V_{\lambda+\lattice}^{\pm}$, $\lambda+\lattice\in \lattice^{\perp}/\lattice$ with $2\lambda\in \lattice$.   
\item
$V_{\lambda+\lattice}\cong V_{-\lambda+\lattice}$, $\lambda+\lattice\in \lattice^{\perp}/\lattice$ with $2\lambda\not\in \lattice$.   
\item
$V_{\lattice}^{T_{\chi},\pm}$ for any irreducible $\hat{\lattice}/P$-module $T_{\chi}$ with central character $\chi$.
\end{enumerate}
\end{theorem}

Here, 
$V_{\lambda+\lattice}^{\pm}=\{\lu\in V_{\lambda+\lattice}\ |\ \theta(\lu)=\pm \lu\}$
for $\lambda+\lattice\in \lattice^{\perp}/\lattice$ with $2\lambda\in \lattice$,
$\hat{\lattice}$ is the canonical central extension of $\lattice$
by the cyclic group $\langle\kappa\rangle$ of order $2$ with  the commutator map
$c(\alpha,\beta)=\kappa^{\langle\alpha,\beta\rangle}$ for $\alpha,\beta\in\lattice$,
$\subL=\{\theta(a) a^{-1}\ |\ a\in\hat{\lattice}\}$,
$V_{\lattice}^{T_{\chi}}$ is a $\theta$-twisted $V_{\lattice}$-module,
and $V_{\lattice}^{T_{\chi},\pm}=\{\lu\in V_{\lattice}^{T_{\chi}}\ |\ \theta(\lu)=\pm \lu\}$.
Note that in Theorem \ref{theorem:classification-weak-module}, 
$V_{\lattice}^{T_{\chi},\pm}$ in (3) are $V_{\lattice}^{+}$-modules,
however, if $\lattice$ is not positive definite,
then $V_{\lambda+\lattice}^{\pm}$ in (1) and $V_{\lambda+\lattice}$ in (2) are not  
$V_{\lattice}^{+}$-modules (see \eqref{eq:dimC(Vlambda+lattice)n=+infty}).
If $\lattice$ is positive definite, then Theorem \ref{theorem:classification-weak-module}
is the same as \cite[Theorem 7.7]{AD2004} and \cite[Theorem 5.13]{DN1999-2}.
Using Theorem \ref{theorem:classification-weak-module},  we show in \cite[Theorem 1.1]{Tanabe-nondeg-rep} that every weak $V_{\lattice}^{+}$-module is completely reducible
for any non-degenerate even lattice $\lattice$ of finite rank.

In this paper, Part 1 of a series of three papers, we show that when the rank of $L$ is $1$,
every non-zero weak $V_{\lattice}^{+}$-module contains a non-zero $M(1)^{+}$-module,
where $M(1)^{+}$ is the fixed point subalgebra of the Heisenberg vertex operator algebra $M(1)$
under the action of $\theta$ (Proposition \ref{proposition:Zhu-Omega}). In Part 2,  we compute extension groups for $M(1)^{+}$-modules.
In Part 3,  after studying intertwining operators for weak $M(1)^{+}$-modules, we show Theorem \ref{theorem:classification-weak-module}.

Let us briefly explain the basic idea to show Theorem \ref{theorem:classification-weak-module} in the case that $\rank L=1$.
We write $L=\Z\alpha$ and we further assume that $\langle\alpha,\alpha\rangle\neq 2$ to simplify the argument.
When $\langle\alpha,\alpha\rangle=2$, we only need to change the set of generators for $V_{\lattice}^{+}$
in the following argument.
Let $V$ be a vertex algebra and $(\module,Y_{\module})$ a weak $V$-module.
For $\lu\in V$ and $\lv\in\module$, we write the expansion
of $Y_{\module}(u,x)v$ by $Y_{\module}(u,x)v=\sum_{i\in\Z}u_ivx^{-i-1}$ 
and define $\epsilon(\lu,\lv)\in\Z\cup\{-\infty\}$ by
\begin{align}
\lu_{\epsilon(\lu,\lv)}\lv&\neq 0
\mbox{ and }\lu_{i}\lv=0\mbox{ for all }i>\epsilon(\lu,\lv)
\end{align}
if $Y_{\module}(u,x)v\neq 0$ and $\epsilon(\lu,\lv)=-\infty$
if $Y_{\module}(u,x)v= 0$.
It is known that the vertex operator algebra $M(1)$ associated to the
Heisenberg algebra is a subalgebra of $V_{\lattice}$
and the fixed point subalgebra $M(1)^{+}=M(1)^{\langle\theta\rangle}$ under the action of $\theta$
is a subalgebra of $V_{\lattice}^{+}$.
The irreducible $M(1)^{+}$-modules 
are classified in \cite{DN1999-1, DN2001}
and the fusion rules for $M(1)^{+}$-modules are determined in \cite{Abe2000, ADL2005}.
The vertex operator algebra $M(1)^{+}$ 
is generated by the conformal vector (Virasoro element) $\omega$ and a homogeneous element $\Har$ (or $J$) of weight $4$,
 and $V_{\lattice}^{+}$ is generated by $M(1)^{+}$ and $E=e^{\alpha}+\theta(e^{\alpha})$ (see \eqref{eq:definition-omega-J-H}).
We can find some relations for $\omega,\Har$ in $M(1)^{+}$ and for $\omega,\Har, E$ in $V_{\lattice}^{+}$ with the help of computer algebra system 
Risa/Asir\cite{Risa/Asir} (Lemma \ref{lemma:relations-M(1)-V(lattice)+}).
Let $\lv$ 
be a non-zero element of a weak $V_{\lattice}^{+}$-module $\module$.
If $\epsilon(\omega,\lv)\geq 1$, then taking suitable actions of the relations on $\lv$
and using the commutation relations,
we obtain relations in 
$\Span_{\C}\{\omega_{\epsilon(\omega,\lv)}^i\Har_{\epsilon(\Har,\lv)}^{j}\lv
\ |\ i,j\in\Z_{\geq 0}\}$ or in
$\Span_{\C}\{\omega_{\epsilon(\omega,\lv)}^i\Har_{\epsilon(\Har,\lv)}^{j}
\ExB_{\epsilon(\ExB,\lv)}\lv\ |\ i,j\in\Z_{\geq 0}\}$
with the help of computer algebra system 
Risa/Asir (Lemmas \ref{lemma:m1+wJ} and \ref{lemma:r=1-s=3}).
Using these relations in $\module$, we can get
a simultaneous eigenvector $\lu\in\module$ for $\omega_1$ and $\Har_{3}$ with
$\epsilon(\omega,\lu)\leq 1$ and $\epsilon(\Har,\lu)\leq 3$, namely
we have an irreducible module $\C \lu$ for the Zhu algebra $A(M(1)^{+})$ of 
$M(1)^{+}$ (Lemmas \ref{lemma:r=1-s=3}, \ref{lemma:structure-Vlattice-M1}, \ref{lemma:structure-Vlattice-M1-norm2}, and \ref{lemma:structure-Vlattice-M1-norm1-2}).
Thus, by \cite[Theorem 6.2]{DLM1998t} we obtain a non-zero $M(1)^{+}$-submodule of $\module$ (Proposition \ref{proposition:Zhu-Omega}).
Moreover, we have some conditions on $\epsilon(\ExB,\lu)$ (Lemmas \ref{lemma:structure-Vlattice-M1}, \ref{lemma:structure-Vlattice-M1-norm2}, and \ref{lemma:structure-Vlattice-M1-norm1-2}).
Using results of extension groups for $M(1)^{+}$-modules (Part 2), we have an irreducible $M(1)^{+}$-module $\mK$ in the $M(1)^{+}$-submodule of $\module$ generated by $\C \lu$ (Part 3).
Since $V_{\lattice}^{+}$ is a direct sum of irreducible
$M(1)^{+}$-modules, for any irreducible $M(1)^{+}$-submodule $\mN$ of $V_{\lattice}^{+}$,
the $V_{\lattice}^{+}$-module structure of $\module$
induces an intertwining operator $I(\mbox{ },x) : \mN\times K\rightarrow\module\db{x}$ for weak $M(1)^{+}$-modules.
By using results of extension groups for $M(1)^{+}$-modules (Part 2),
the same argument as above shows that 
there exists an $M(1)^{+}$-module which is a
direct sum of irreducible $M(1)^{+}$-modules in the image of $I(\mbox{ },x)$ (Part 3).
Thus, we obtain a weak irreducible $V_{\lattice}^{+}$-submodule $\mW$ of $\module$
which is isomorphic to a submodule of a $\theta$-twisted irreducible $V_{\lattice}$-module or 
is a direct sum of  pairwise non-isomorphic irreducible $M(1)^{+}$-modules (Part 3).
In the latter case,
by a standard argument we can determine the possible weak $V_{\lattice}^{+}$-module structures for such an $M(1)^{+}$-module (Part 3)
and hence we obtain the desired result.
For general $\lattice$,
we divide our analysis into four cases based on the norm of an element of $\C\otimes_{\Z}\lattice$
and carry out the procedure above by an enormous amount of computation.

Complicated computation has been done by a computer algebra system Risa/Asir\cite{Risa/Asir}.
Throughout this paper, the word \lq\lq a direct computation\rq\rq \ means a direct computation with the help of Risa/Asir.

The organization of the paper is as follows. 
In Section \ref{section:preliminary} we recall some basic properties of the
vertex algebra $M(1)$ associated to the Heisenberg algebra and 
the vertex algebra $V_{\lattice}$ associated to a non-degenerate even lattice $\lattice$.
In Section \ref{section:Modules for the Zhu algera of}
under the condition that the rank of $L$ is $1$,
we construct an irreducible module for the Zhu algebra
of $M(1)^{+}$ in a non-zero weak $V_{\lattice}^{+}$-module $\module$.
Thus, there is a non-zero $M(1)^{+}$-module in $\module$.
In \ref{section:appendix} we put 
computations of $a_{k}b$ for some $a,b\in V_{\lattice}^{+}$
and $k=0,1,\ldots$ to find the commutation relation $[a_{i},b_{j}]=\sum_{k=0}^{\infty}\binom{i}{k}(a_{k}b)_{i+j-k}$. 
In \ref{section:notation} we list some notation.

\section{\label{section:preliminary}Preliminary}
We assume that the reader is familiar with the basic knowledge on
vertex algebras as presented in \cite{B1986,FLM,LL,Li1996}. 

Throughout this paper, $\mn$ is a non-zero complex number,
$\N$ denotes the set of all non-negative integers,
$\Z$ denotes the set of all integers, 
$\lattice$ is a non-degenerate even lattice of finite rank $\rankL$ with a bilinear form $\langle\ ,\ \rangle$,
$(V,Y,{\mathbf 1})$ is a vertex algebra.
Recall that $V$ is the underlying vector space, 
$Y(\mbox{ },\wx)$ is the linear map from $V\otimes_{\C}V$ to $V\db{x}$, and
${\mathbf 1}$ is the vacuum vector.
Throughout this paper, we always assume that $V$ has an element $\omega$ such that $\omega_{0}a=a_{-2}\vac$ for all $a\in V$.
For a vertex operator algebra $V$, this condition automatically holds since $V$ has the conformal vector (Virasoro element).
For $i\in\Z$, we define
\begin{align}
	\Z_{< i}&=\{j\in\Z\ |\ j< i\}\mbox{ and }\Z_{> i}=\{j\in\Z\ |\ j> i\}.
\end{align}
The notion of a module for $V$ has been introduced in several papers, however, if $V$ is a vertex operator algebra,
then the notion of a module for $V$ viewed as a vertex algebra is different from the notion of a module for $V$ viewed as 
a vertex operator algebra (cf. \cite[Definitions 4.1.1 and 4.1.6]{LL}).
To avoid confusion, throughout this paper, we refer to a module for a vertex operator algebra defined in \cite[Definition 4.1.6]{LL} as a {\it module}
and to a module for a vertex algebra defined in \cite[Definition 4.1.1]{LL} as a {\it weak module}.
The reason why we use the terminology \lq\lq weak module\rq\rq\ is that when $V$ is a vertex operator algebra, a module for $V$ viewed as a vertex algebra is called 
a weak $V$-module (cf. \cite[p.157]{Li1996}, \cite[p.150]{DLM1997}, and \cite[Definition 2.3]{ABD2004}).
We write down the definition of a weak $V$-module:
\begin{definition}
\label{definition:weak-module}
A {\it weak $V$-module} $\module$ is a vector space over $\C$ equipped with a linear map
\begin{align}
\label{eq:inter-form}
Y_{\module}(\ , x) : V\otimes_{\C}\module&\rightarrow \module\db{x}\nonumber\\
a\otimes u&\mapsto  Y_{\module}(a, x)\lu=\sum_{n\in\Z}a_{n}\lu x^{-n-1}
\end{align}
such that the following conditions are satisfied:
\begin{enumerate}
\item $Y_{\module}(\vac,x)=\id_{\module}$.
\item
For $a,b\in V$ and $\lu\in \module$,
\begin{align}
\label{eq:inter-borcherds}
&x_0^{-1}\delta(\dfrac{x_1-x_2}{x_0})Y_{\module}(a,x_1)Y_{\module}(b,x_2)\lu-
x_0^{-1}\delta(\dfrac{x_2-x_1}{-x_0})Y_{\module}(b,x_2)Y_{\module}(a,x_1)\lu\nonumber\\
&=x_1^{-1}\delta(\dfrac{x_2+x_0}{x_1})Y_{\module}(Y(a,x_0)b,x_2)\lu.
\end{align}
\end{enumerate}
\end{definition}
For $n\in\C$ and a weak $V$-module $\module$, we define $M_{n}=\{\lu\in V\ |\ \omega_1 \lu=n \lu\}$.
For $a\in V_{n}\ (n\in\C)$, $\wt a$ denotes $n$.
For a vertex algebra $V$ which admits a decomposition $V=\oplus_{n\in\Z}V_n$ and a subset $U$ of a weak $V$-module, we 
define
\begin{align}
\label{eq:OmegaV(U)=BigluinU}
\Omega_{V}(U)&=\Big\{\lu\in U\ \Big|\ 
\begin{array}{l}
a_{i}\lu=0\ \mbox{for all homogeneous }a\in V\\
\mbox{and }i>\wt a-1.
\end{array}\Big\}.
\end{align}
For a vertex algebra $V$ which admits a decomposition $V=\oplus_{n\in\Z}V_n$, a weak $V$-module $\mN$
 is called {\it $\N$-graded} if $N$ admits a decomposition $N=\oplus_{n=0}^{\infty}N(n)$
such that $a_{i}N(n)\subset N(\wt a-i-1+n)$ for all homogeneous $a\in V$, $i\in\Z$, and $n\in\Z_{\geq 0}$, where 
we define $N(n)=0$ for all $n<0$. For a triple of weak $V$-modules $\module, \mN,\mW$,
$\lu\in\module, \lv\in\mW$, and an intertwining operator $I(\ ,x)$ from $\module\times \mW$ to $\mN$, 
we write the expansion of $I(u,x)v$ by
\begin{align}
I(\lu,x)\lv
&=\sum_{i\in\C}\lu_i\lv x^{-i-1}
=\sum_{i\in\C}I(\lu;i)\lv x^{-i-1}\in \mN\{x\}.
\end{align}
In this paper, we consider only the case that the image of $I(\mbox{ },x)$ is contained in
$\mN\db{x}$,
namely $I(\ ,x) : \module\times \mW\rightarrow \mN\db{x}$. 
For a subset $X$ of $\mW$,
\begin{align}
M\cdot X\mbox{ denotes }\Span_{\C}\{a_{i}\lu\ |\ a\in \module, i\in\Z, \lu\in X\}\subset N. 
\end{align}
For an intertwining operator $I(\mbox{ },x) : \module\times \mW\rightarrow \mN\db{x}$,
$\lu\in\module$, and $\lv\in \mW$, we define  $\epsilon_{I}(\lu,\lv)=\epsilon(\lu,\lv)\in\Z\cup\{-\infty\}$ by
\begin{align}
\label{eqn:max-vanish}
\lu_{\epsilon_{I}(\lu,\lv)}\lv&\neq 0\mbox{ and }\lu_{i}\lv
=0\mbox{ for all }i>\epsilon_{I}(\lu,\lv)
\end{align}
if $I(\lu,x)\lv\neq 0$ and $\epsilon_{I}(\lu,\lv)=-\infty$
if $I(\lu,x)\lv= 0$.
If $V$ is a simple vertex algebra and 
$(\module, Y_{\module})$ is a weak $V$-module, then 
it follows from \cite[Corollary 4.5.15]{LL} that
\begin{align}
\label{eq:epsilonI(u,v)inZ}
Y_{\module}(a,x)u\neq 0, \mbox{ namely }
\epsilon_{Y_{\module}}(a,u)\in\Z
\end{align}  
for any non-zero $a\in V$ and any non-zero $u\in M$.
We will frequently use the following easy formula:
\begin{lemma}
\label{lemma:comm-change}
Let $\module, \mW, \mN$ be three weak $V$-modules
and  $I(\ ,x) : \module\times \mW\rightarrow \mN\db{x}$ an intertwining operator. 
For $a\in V$, $b\in \module$, $m\in\Z_{\geq -1}$, $k\in\Z_{\leq -1}$, and $n\in\Z$,
\begin{align}
\label{eq:abm}
(a_{k}b)_{n}
&=\sum_{\begin{subarray}{l}i\leq m\\i+j-k=n\end{subarray}}\binom{-i-1}{-k-1}a_{i}b_{j}+\sum_{\begin{subarray}{l}i\geq m+1\\i+j-k=n\end{subarray}}\binom{-i-1}{-k-1}b_{j}a_{i}\nonumber\\
&\quad{}-\sum_{i=0}^{m}\binom{-i-1}{-k-1}\sum_{\wl=0}^{\infty}\binom{i}{\wl}(a_{\wl}b)_{n+k-\wl}\nonumber\\
&=\sum_{\begin{subarray}{l}i\leq m\\i+j-k=n\end{subarray}}
\binom{-i-1}{-k-1}a_{i}b_{j}+\sum_{\begin{subarray}{l}i\geq m+1\\i+j-k=n\end{subarray}}\binom{-i-1}{-k-1}b_{j}a_{i}\nonumber\\
&\quad{}+(-1)^{k}\sum_{\wl=0}^{\infty}\binom{\wl-k-1}{-k-1}
\binom{m-k}{\wl-k}(a_{\wl}b)_{n+k-\wl}.
\end{align}
\end{lemma}
\begin{proof}
The first expression follows from
\begin{align}
(a_{k}b)_{n}
&=\sum_{\begin{subarray}{l}i<0\\i+j-k=n\end{subarray}}\binom{-i-1}{-k-1}a_{i}b_{j}+\sum_{\begin{subarray}{l}i\geq 0\\i+j-k=n\end{subarray}}\binom{-i-1}{-k-1}b_{j}a_{i}\nonumber\\
&=\sum_{\begin{subarray}{l}i\leq m\\i+j-k=n\end{subarray}}\binom{-i-1}{-k-1}a_{i}b_{j}+\sum_{\begin{subarray}{l}i\geq m+1\\i+j-k=n\end{subarray}}\binom{-i-1}{-k-1}b_{j}a_{i}
-\sum_{\begin{subarray}{l}0\leq i\leq m\\i+j-k=n\end{subarray}}\binom{-i-1}{-k-1}[a_{i},b_{j}].
\end{align}
The last expression follows from the fact that
$\sum_{i=0}^{m}\binom{-i-1}{-k-1}\binom{i}{\wl}
=(-1)^{k+1}\binom{\wl-k-1}{-k-1}
\binom{m-k}{\wl-k}$ for $l\in\Z_{\geq 0}$.
\end{proof}

We recall the vertex operator algebra $M(1)$ associated to the Heisenberg algebra and 
the vertex algebra $V_{\lattice}$ associated to a non-degenerate even lattice $\lattice$.
Let $\hei$ be a finite dimensional vector space equipped with a non-degenerate symmetric bilinear form
$\langle \mbox{ }, \mbox{ }\rangle$.
Set a Lie algebra
\begin{align}
\hat{\hei}&=\hei\otimes \C[t,t^{-1}]\oplus \C \centralc
\end{align} 
with the Lie bracket relations 
\begin{align}
[\beta\otimes t^{m},\gamma\otimes t^{n}]&=m\langle \beta,\gamma\rangle\delta_{m+n,0}\centralc,&
[\centralc,\hat{\hei}]&=0
\end{align}
for $\beta,\gamma\in \hei$ and $m,n\in\Z$.
For $\beta\in \hei$ and $n\in\Z$, $\beta(n)$ denotes $\beta\otimes t^{n}\in\widehat{H}$. 
Set two Lie subalgebras of $\fh$:
\begin{align}
\widehat{\hei}_{{\geq 0}}&=\bigoplus_{n\geq 0}\hei \otimes t^n\oplus \C \centralc&\mbox{ and }&&
\widehat{\hei}_{<0}&=\bigoplus_{n\leq -1}\hei\otimes t^n.
\end{align}
For $\beta\in\fh$,
$\C e^{\beta}$ denotes the one dimensional $\widehat{\fh}_{\geq 0}$-module uniquely determined
by the condition that for $\gamma\in\fh$
\begin{align}
\gamma(i)\cdot e^{\beta}&
=\left\{
\begin{array}{ll}
\langle\gamma,\beta\rangle e^{\beta}&\mbox{ for }i=0\\
0&\mbox{ for }i>0
\end{array}
\right.&&\mbox{ and }&
\centralc\cdot e^{\beta}&=e^{\beta}.
\end{align}
We take an $\widehat{\fh}$-module 
\begin{align}
\label{eq:untwist-induced}
M(1,\beta)&=\sU (\widehat{\fh})\otimes_{\sU (\widehat{\fh}_{\geq 0})}\C e^{\beta}
\cong \sU(\widehat{\fh}_{<0})\otimes_{\C}\C e^{\beta}
\end{align}
where $\sU(\fg)$ is the universal enveloping algebra of a Lie algebra $\fg$.
Then, $M(1)=M(1,0)$ has a vertex operator algebra structure with
the conformal vector
\begin{align}
\label{eq:conformal-vector}
\omega&=\dfrac{1}{2}\sum_{i=1}^{\dim\fh}h_i(-1)h_i^{\prime}(-1)\vac
\end{align}
where $\{h_1,\ldots,h_{\dim\fh}\}$ is a basis of $\fh$ and
$\{h_1^{\prime},\ldots,h_{\dim\fh}^{\prime}\}$ is its dual basis.
Moreover, $M(1,\beta)$ is an irreducible $M(1)$-module for any $\beta\in\fh$. 
The vertex operator algebra $M(1)$ is called the {vertex operator algebra associated to
 the Heisenberg algebra} $\oplus_{0\neq n\in\Z}\fh\otimes t^{n}\oplus \C \centralc$. 

Let $\lattice$ be a non-degenerate even lattice.
We define $\fh=\C\otimes_{\Z}\lattice$
and denote by $\lattice^{\perp}$ the dual of $\lattice$: $\lattice^{\perp}=\{\gamma\in\fh\ |\ \langle\beta,\gamma\rangle\in\Z\mbox{ for all }\beta\in\lattice\}$. 
Taking $M(1)$ for $\fh$,
we define $V_{\lambda+\lattice}=\oplus_{\beta\in\lambda+\lattice}M(1,\beta)$
for $\lambda+\lattice\in \lattice^{\perp}/\lattice$.
Then, $V_{\lattice}$ admits a unique vertex algebra structure compatible with the action of $M(1)$ and 
for each $\lambda+\lattice\in\lattice^{\perp}/\lattice$
the vector space
$V_{\lambda+\lattice}$ is an irreducible weak $V_{\lattice}$-module which admits the following decomposition:
\begin{align}
V_{\lambda+\lattice}&=\bigoplus_{n\in\langle\lambda,\lambda\rangle/2+\Z}(V_{\lambda+\lattice})_{n} \mbox{ where }
(V_{\lambda+\lattice})_{n}=\{a\in V_{\lambda+\lattice}\ |\ \omega_{1}a=na\}.
\end{align}
Note that if $\lattice$ is positive definite, then $\dim_{\C}(V_{\lambda+\lattice})_{n}<+\infty$
for all $n\in \lambda+\lattice$
and $(V_{\lambda+\lattice})_{\langle\lambda,\lambda\rangle/2+i}=0$ for sufficiently small $i\in\Z$.
If $\lattice$ is not positive definite, then  
\begin{align}
\label{eq:dimC(Vlambda+lattice)n=+infty}
\dim_{\C}(V_{\lambda+\lattice})_{n}=+\infty
\end{align}
for all $n\in \langle\lambda,\lambda\rangle/2+\Z$, which implies that $V_{\lambda+\lattice}$ is not a $V_{\lattice}$-module.
For $\alpha\in \fh$, we write
\begin{align}
E(\alpha)&=e^{\alpha}+\theta(e^{\alpha})
\end{align}

Let  $\hat{\lattice}$ be the canonical central extension of $\lattice$
by the cyclic group $\langle\kappa\rangle$ of order $2$ with  the commutator map
$c(\alpha,\beta)=\kappa^{\langle\alpha,\beta\rangle}$ for $\alpha,\beta\in\lattice$:
\begin{align}
	0\rightarrow \langle\kappa\rangle\overset{}{\rightarrow} \hat{\lattice}\overset{-}{\rightarrow} \lattice\rightarrow 0.
\end{align}
Then, the $-1$-isometry of $\lattice$ induces an automorphism $\theta$ of $\hat{\lattice}$ of order $2$
and an automorphism, by abuse of notation we also denote by $\theta$, of $V_{\lattice}$ of order $2$.
In $M(1)$, we have
\begin{align}
\label{eq:theta}
\theta(h^1(-i_1)\cdots h^n(-i_n)\vac)&=(-1)^{n}h^1(-i_1)\cdots h^n(-i_n)\vac
\end{align}
for $n\in\Z_{\geq 0}$, $h^1,\ldots,h^{n}\in\fh$, and $i_1,\ldots,i_n\in\Z_{>0}$.
For a weak $V_{\lattice}$-module $\module$,
we define a weak $V_{\lattice}$-module $(\module\circ \theta,Y_{\module\circ \theta})$
by $\module\circ\theta=\module$ and 
\begin{align}
Y_{\module\circ \theta}(a,x)&=Y_{\module}(\theta(a),x)
\end{align}
for  $a\in V_{\lattice}$.
Then
$V_{\lambda+\lattice}\circ\theta\cong V_{-\lambda+\lattice}$
for $\lambda\in \lattice^{\perp}$.
Thus, for $\lambda\in \lattice^{\perp}$ with $2\lambda\in\lattice$
we define
\begin{align}
V_{\lambda+\lattice}^{\pm}&=\{u\in V_{\lambda+\lattice}\ |\ \theta(u)=\pm u\}.
\end{align}
Next, we recall the construction of $\theta$-twisted modules for $M(1)$ and $V_{\lattice}$ following \cite{FLM}.
Set a Lie algebra
\begin{align}
\hat{\hei}[-1]&=\hei\otimes t^{1/2}\C[t,t^{-1}]\oplus \C \centralc
\end{align} 
with the Lie bracket relations 
\begin{align}
[\centralc,\hat{\hei}[-1]]&=0&\mbox{and}&&
[\alpha\otimes t^{m},\beta\otimes t^{n}]&=m\langle\alpha,\beta\rangle\delta_{m+n,0}\centralc
\end{align}
for $\alpha,\beta\in \hei$ and $m,n\in1/2+\Z$.
For $\alpha\in \hei$ and $n\in1/2+\Z$, $\alpha(n)$ denotes $\alpha\otimes t^{n}\in\widehat{\hei}$. 
Set two Lie subalgebras of $\hat{\hei}[-1]$:
\begin{align}
\widehat{\hei}[-1]_{{\geq 0}}&=\bigoplus_{n\in 1/2+\N}\hei\otimes t^n\oplus \C \centralc&\mbox{ and }&&
\widehat{\hei}[-1]_{{<0}}&=\bigoplus_{n\in 1/2+\N}\hei\otimes t^{-n}.
\end{align}
Let $\C \vac_{\tw}$ denote a unique one dimensional $\widehat{\hei}[-1]_{{\geq 0}}$-module 
such that 
\begin{align}
h(i)\cdot \vac_{\tw}&
=0\quad\mbox{ for }h\in\fh\mbox{ and }i\in \frac{1}{2}+\N,\nonumber\\
\centralc\cdot \vac_{\tw}&=\vac_{\tw}.
\end{align}
We take an $\widehat{\hei}[-1]$-module 
\begin{align}
\label{eq:twist-induced}
M(1)(\theta)
&=\sU (\widehat{\hei}[-1])\otimes_{\sU (\widehat{\hei}[-1]_{\geq 0})}\C u_{\ul{\zeta}}
\cong\sU (\widehat{\hei}[-1]_{<0})\otimes_{\C}\C u_{\ul{\zeta}}.
\end{align}
We define for $\alpha\in \hei$, 
\begin{align}
	\alpha(x)&=\sum_{i\in 1/2+\Z}\alpha(i)x^{-i-1}
	\end{align}
and for $u=\alpha_1(-\wi_1)\cdots \alpha_k(-\wi_k)\vac\in M(1)$, 
\begin{align}
	Y_{0}(u,x)&=\nor
\dfrac{1}{(\wi_1-1)!}	(\dfrac{d^{\wi_1-1}}{dx^{\wi_1-1}}\alpha_1(x))
	\cdots
\dfrac{1}{(\wi_k-1)!}	(\dfrac{d^{\wi_k-1}}{dx^{\wi_k-1}}\alpha_k(x))\nor.
\end{align}
Here, for $\beta_1,\ldots,\beta_{n}\in \fh$ and $i_1,\ldots,i_n\in1/2+\Z$, we define 
$\nor \beta_1(i_1)\cdots\beta_{n}(i_n)\nor$ inductively by
\begin{align}
\label{eq:nomal-ordering}
\nor \beta_1(i_1)\nor&=\beta_1(i_1)\qquad\mbox{ and}\nonumber\\
\nor \beta_1(i_1)\cdots\beta_{n}(i_n)\nor&=
\left\{
\begin{array}{ll}
\nor \beta_{2}(i_2)\cdots\beta_{n}(i_n)\nor \beta_1(i_1)&\mbox{if }i_1\geq 0,\\
\beta_{1}(i_1)\nor \beta_{2}(i_2)\cdots\beta_{n}(i_n)\nor &\mbox{if }i_1<0.
\end{array}\right.
\end{align}
Let $h^{[1]},\ldots,h^{[\dim\fh]}$ be an orthonormal basis of $\fh$.
We define $c_{mn}\in\Q$ for $ m,n\in \Z_{\geq 0}$ by
\begin{align}
	\sum_{m,n=0}^{\infty}c_{mn}x^{m}y^{n}&=-\log (\dfrac{(1+x)^{1/2}+(1+y)^{1/2}}{2})
	\end{align}
and
	\begin{align}
		\Delta_{x}&=\sum_{m,n=0}^{\infty}c_{mn}\sum_{i=1}^{\dim\fh}h^{[i]}(m)h^{[i]}(n)x^{-m-n}.
		\end{align}
Then, for $u\in M(1)$ we define a vertex operator $Y_{M(1)(\theta)}$ by
\begin{align}
Y_{M(1)(\theta)}(u,x)&=Y_{0}(e^{\Delta_{x}}u,x).
\end{align}
Then, \cite[Theorem 9.3.1]{FLM}
shows that 
$(M(1)(\theta),Y_{M(1)(\theta)})$ is an irreducible $\theta$-twisted $M(1)$-module.
Set a submodule $\subL=\{\theta(a) a^{-1}\ |\ a\in\hat{\lattice}\}$ of $\hat{\lattice}$.
Let $T_{\chi}$ be the irreducible $\hat{L}/\subL$-module associated to a
central character $\chi$ such that $\chi(\kappa)=-1$.
We set
\begin{align}
V_{\lattice}^{T_{\chi}}&=M(1)(\theta)\otimes T_{\chi}.
\end{align} 
Then, \cite[Theorem 9.5.3]{FLM} shows that 
$V_{\lattice}^{T_{\chi}}$ admits an irreducible $\theta$-twisted $V_{\lattice}$-module structure compatible with the action of $M(1)$.
We define the action of $\theta$ on $V_{\lattice}^{T_{\chi}}$ by
\begin{align}
\theta(h^1(-i_1)\cdots h^n(-i_n)\lu)&=(-1)^{n}h^1(-i_1)\cdots h^n(-i_n)\lu
\end{align}
for $n\in\Z_{\geq 0}$, $h^1,\ldots,h^{n}\in\fh$, $i_1,\ldots,i_n\in 1/2+\Z_{>0}$, and $\lu\in T_{\chi}$.
We set
\begin{align}
V_{\lattice}^{T_{\chi},\pm}&=\{\lu\in V_{\lattice}^{T_{\chi}}\ |\ \theta(u)=\pm u\}.
\end{align}

We recall the {\it Zhu algebra} $A(V)$ of a vertex operator algebra $V$ from \cite[Section 2]{Z1996}.
For homogeneous $a\in V$ and $b\in V$, we define
\begin{align}
\label{eq:zhu-ideal-multi}
a\circ b&=\sum_{i=0}^{\infty}\binom{\wt a}{i}a_{i-2}b\in V
\end{align}
and 
\begin{align}
\label{eq:zhu-bimodule-left}
a*b&=\sum_{i=0}^{\infty}\binom{\wt a}{i}a_{i-1}b\in V.
\end{align}
We extend \eqref{eq:zhu-ideal-multi} and \eqref{eq:zhu-bimodule-left} for an arbitrary $a\in V$ by linearity.
We also define
$O(V)=\Span_{\C}\{a\circ b\ |\ a,b\in V\}$.
Then, the quotient space
\begin{align}
\label{eq:zhu-bimodule}
A(V)&=M/O(V)\textcolor{red}{,}
\end{align}
called the {\em Zhu algebra} of $V$, is an associative $\C$-algebra with multiplication  
\eqref{eq:zhu-bimodule-left} by \cite[Theorem 2.1.1]{Z1996}.

\section{Modules for the Zhu algebra of $M(1)^{+}$ in a weak $V_{\lattice}^{+}$-module:
the case that $\rank L=1$}
\label{section:Modules for the Zhu algera of}

In this section, under the condition that the rank of $L$ is $1$,
we shall show that there exists an irreducible $A(M(1)^{+})$-module in an arbitrary non-zero 
weak $V_{\lattice}^{+}$-module.

Throughout this section, $\mn$ is a non-zero complex number, 
$\hei$ is a one dimensional vector space equipped with a non-degenerate symmetric bilinear form
$\langle \mbox{ }, \mbox{ }\rangle$, $\wh,\alpha\in\hei$ such that
\begin{align}
\langle \wh,\wh\rangle&=1\mbox{ and }
\langle\alpha,\alpha\rangle=\mn,
\end{align}
$\module,\mN,\mW$ are weak $M(1)^{+}$-modules,
and 
$I(\mbox{ },x) : M(1,\alpha)\times \mW\rightarrow N\db{x}$ is  a non-zero intertwining operator.
We define
\begin{align}
\omega&=\dfrac{1}{2}h(-1)^2\vac,\nonumber\\
\Har&=\dfrac{1}{3}h(-3)h(-1)\vac-\dfrac{1}{3}h(-2)^2\vac,\nonumber\\
J&=h(-1)^4\vac-2h(-3)h(-1)\vac+\dfrac{3}{2}h(-2)^2\vac\nonumber\\
&=-9\Har+4\omega_{-1}^2\vac-3\omega_{-3}\vac,\nonumber\\
\ExB&=\ExB(\alpha)=e^{\alpha}+\theta(e^{\alpha}).
\label{eq:definition-omega-J-H}
\end{align}
Since \begin{align}
	0&=\omega_{-2}h(-1)\vac-2\omega_{-1}h(-2)\vac+3h(-4)\vac\label{eq:w-2h-2wh3h}
\end{align}
and $[\omega_{0},a_{i}]=-ia_{i-1}$ for all $a\in M(1)$ and $i\in\Z$,
\begin{align}
\label{eq:h(i)vacinSpanBig}
h(j)\vac\in \Span\Big\{\omega_{-i_1}\cdots\omega_{-i_m}h(-k)\vac\ \Big|\ 
\begin{array}{l}m\in\Z_{\geq 0},i_1,\ldots,i_m\in\Z_{>0}\\
\mbox{ and }k=1,2,3\end{array}\Big\}
\end{align}
for all $j\in\Z$. 
For $i,j\in\Z$, a direct computation shows that
\begin{align}
	[\omega_{i},\omega_{j}]&=(i-j)\omega_{i+j-1}+\delta_{i+j-2,0}\dfrac{i(i-1)(i-2)}{12},\label{eq:wiwj}\\
	[\omega_{i},J_{j}]&=(3i-j)J_{i+j-1},\label{eq:wiJj}\\
	[\omega_i,\Har_{j}]&
	=(3i-j)\Har_{i+j-1}
	+\frac{i(i-1)(3i+j-6)}{6}\omega_{i+j-3}\nonumber\\
	&\quad{}+\dfrac{-1}{3}\binom{i}{5}\delta_{i+j-4,0},\label{eq:wiHj}\\
	[\omega_{i},\ExB_{j}]&=((-1+\dfrac{\mn}{2})i-j)\ExB_{i+j-1}, \label{eq:wiEji}\\
	[h(i),\omega_{j}]&=h(i+j-1),\label{eq:hiomegaj}\\
	[h(i),\Har_{j}]
	&=\big(\frac{i(i+j-2)(5i+j-5)}{6}+\binom{i}{3}\big)h(i+j-3),
\end{align}
and
\begin{align}
\label{eq:h(-2)h(-1)=omega0} 
	h(-2)h(-1)\vac&=
	\omega_{0 } \omega,\nonumber\\
	h(-3)h(-1)\vac&=
	\Har
	+\frac{1}{3}\omega_{0 }^{2}\omega,\nonumber\\
	h(-2)h(-2)\vac&=
	-2\Har
	+\frac{1}{3}\omega_{0 }^{2}\omega,\nonumber\\
	h(-3)h(-2)\vac&=\frac{-1}{2}\omega_{0 } \Har
	+\frac{1}{12}\omega_{0 }^{3}\omega,\nonumber\\
	h(-3)h(-3)\vac&=
	\frac{1}{3}\omega_{-2 }^2\vac
	+\frac{4}{5}\omega_{-1 } \Har
	+\frac{-1}{15}\omega_{0 }^{2}\omega_{-1 } \omega
	+\frac{-3}{10}\omega_{0 }^{2}\Har
	+\frac{1}{45}\omega_{0 }^{4}\omega.
\end{align}
It follows from \eqref{eq:h(i)vacinSpanBig}, \eqref{eq:wiHj}, and \eqref{eq:hiomegaj}--\eqref{eq:h(-2)h(-1)=omega0} 
that $M(1)^{+}$ is spanned by the elements 
$\omega_{-i_1}\cdots \omega_{-i_m}\Har_{-j_1}\cdots \Har_{-j_n}\vac$
where $m,n\in\Z_{\geq 0}$ and $i_1,\ldots,i_m,j_1,\ldots,j_n\in\Z_{>0}$,
which is already shown in \cite[Theorem 2.7]{DG1998}.

We have
\begin{align}
\label{eq:J1E=frac2}
J_{2}\ExB&=2(2\mn-1)\omega_{0 } \ExB,\quad
J_{3}\ExB=(\mn^2-\dfrac{\mn}{2})\ExB,\quad
J_{i}\ExB=0\mbox{ for }i\geq 4,
\end{align}
and 
\begin{align}
\label{eq:Har1ExB=dfrac2mn2mn1omega1ExB}
\Har_{2}\ExB&=\dfrac{1}{3}\omega_{0 } \ExB,\quad
\Har_{i}\ExB=0\mbox{ for }i\geq 3.
\end{align}
If $\mn\neq 1/2$, then
\begin{align}
\label{eq:J1ExB=frac2mn(4mn-11)-1}
J_{1}\ExB&=\frac{2\mn(4\mn-11)}{2\mn-1}\omega_{-1 } \ExB
+\frac{8\mn+5}{2\mn-1}\omega_{0 }^2 \ExB,\nonumber\\
\Har_{1}\ExB&=
\dfrac{2\mn}{2\mn-1}\omega_{-1 } \ExB
+\dfrac{-1}{2\mn-1}\omega_{0 }^2 \ExB
\end{align}
and if $\mn\neq 2,1/2$, then
\begin{align}
\label{eq:J1ExB=frac2mn(4mn-11)-2}
J_{0}\ExB&=\frac{2 (\mn-8) (2 \mn-1)}{\mn-2}\omega_{-2 } \ExB
+\frac{4 (4 \mn^2-\mn+4)}{(\mn-2) (2 \mn-1)}\omega_{0 } \omega_{-1 } \ExB
\nonumber\\&\quad{}
+\frac{-18}{(\mn-2) (2 \mn-1)}\omega_{0 }^3 \ExB\nonumber,\\
\Har_{0}\ExB&=
\dfrac{2\mn}{\mn-2}\omega_{-2 } \ExB
+\dfrac{-4\mn}{(2\mn-1)(\mn-2)}\omega_{0 } \omega_{-1 } \ExB
+\dfrac{2}{(2\mn-1)(\mn-2)}\omega_{0 }^3 \ExB.
\end{align}
If $\mn=2$, then by Lemma \ref{lemma:comm-change} and \eqref{eq:omega[1]0Har[1]0ExB=omega0} in \ref{section:appendix} 
for $m\geq 0$,
\begin{align}
\label{eq:[omegai(Har0ExB)j]}
&		[\omega_{i},(\Har_{0}\ExB)_{j}]\nonumber\\
		&=(3i-j)(\Har_{0}\ExB)_{i+j-1}\nonumber\\
		&\quad{}+8\binom{i}{2}\big( 
		\sum_{k\leq m}\omega_{k}\ExB_{i+j-3-k}+\sum_{k\geq m+1}\ExB_{i+j-3-k}\omega_{k }\nonumber\\
&\qquad{}+((m+1)(i+j-3)-\binom{m+1}{2})\ExB_{i+j-4}\big)\nonumber\\
		&\quad{}-2\binom{i}{2}(i+j-2)(i+j-3) \ExB_{i+j-4}\nonumber\\
		&\quad{}-6\binom{i}{3}(i+j-3)\ExB_{i+j-4}+12\binom{i}{4}\ExB_{i+j-4}.
\end{align}
If $\mn=1/2$, then by \eqref{eq:norm1-2-omega[1]0Har[1]1ExB} in \ref{section:appendix},
\begin{align}
\label{eq:[omegaiHar1ExBj]}
[\omega_{i},(\Har_{1}\ExB)_{j}]
&=(\frac{5}{4}i-j)(\Har_{1}\ExB)_{i+j-1}\nonumber\\
&\quad{}+\binom{i}{2}(i+j-2)\ExB_{i+j-3}+
\binom{i}{3}\ExB_{i+j-3}.
\end{align}
For $n\in\Z$, $m\in\Z_{\geq -1}$, and $k\in\Z_{<0}$,
using Lemma \ref{lemma:comm-change} and \eqref{eq:wiEji},
we expand each of $(\omega_{k}E)_{\wn}$ and $(\omega_{-1}^2E)_{n}$
so that 
the resulting expression is  a linear combination of elements of the form 
\begin{align}
\omega_{i_1}\cdots \omega_{i_r}\ExB_{l}\omega_{j_1}\cdots \omega_{j_s}
\end{align}
where $r,s\in\Z_{\geq 0}$, $l\in\Z$, 
$i_1,\ldots,i_r\leq m
$,
and $j_1,\ldots,j_s\geq m+1$
as follows:
\begin{align}
\label{eq:(omega-12E)wn-1}
&(\omega_{k}E)_{\wn}\nonumber\\
&=\sum_{i\leq m}\binom{-i-1}{-k-1}\omega_{i}E_{\wn+k-i}+\sum_{i\geq m+1}\binom{-i-1}{-k-1}E_{\wn+k-i}\omega_{i}\nonumber\\
&\quad{}+(-1)^{k}((-\wn-k)\binom{m-k}{-k}+
\binom{-k}{-k-1}\binom{m-k}{1-k}\dfrac{\mn}{2})E_{\wn+k-1}\nonumber\\
\end{align}
and
\begin{align}
	\label{eq:(omega-12E)wn-2}
&	(\omega_{-1}^2E)_{n}\nonumber\\
&=
\sum_{\begin{subarray}{l}i<0,j<0,\\k=n-i-j-2\end{subarray}}\omega_{i}\omega_{j}E_{k}
+	2\sum_{\begin{subarray}{l}i<0,0\leq j,\\k=n-i-j-2\end{subarray}}\omega_{i}E_{k}\omega_{j}
	+\sum_{\begin{subarray}{l}0\leq i,0\leq j,\\k=n-i-j-2\end{subarray}}E_{k}\omega_{j}\omega_{i}\nonumber\\
&=\sum_{\begin{subarray}{l}i<0,j<0,\\k=n-i-j-2\end{subarray}}\omega_{i}\omega_{j}E_{k}\nonumber\\
&\quad{}+	2\sum_{\begin{subarray}{l}i<0,0\leq j\leq m,\\k=n-i-j-2\end{subarray}}
\big(\omega_{i}\omega_{j}E_{k}
-((-1+\frac{\mn}{2})j-k)\omega_{i}E_{j+k-1}\big)+2\sum_{\begin{subarray}{l}i<0,m+1\leq j,\\k=n-i-j-2\end{subarray}}\omega_{i}E_{k}\omega_{j}\nonumber\\
&\quad{}+\sum_{\begin{subarray}{l}0\leq i,j\leq m,\\k=n-i-j-2\end{subarray}}\Big(
\omega_{j}\omega_{i}E_{k}+
((1-\frac{p}{2})j+k)((1-\frac{p}{2})i+j+k-1)E_{j+k-2}
\nonumber\\
&\qquad{}+((1-\frac{p}{2})j+k)\omega_{i}E_{j+k-1}+((1-\frac{p}{2})i+k)\omega_{j}E_{i+k-1}\Big)\nonumber\\
&\quad{}+\sum_{\begin{subarray}{l}m+1\leq i,0\leq j\leq m,\\k=n-i-j-2\end{subarray}}\big(((1-\frac{p}{2})j+k)E_{j+k-1}\omega_{i}+\omega_{j}E_{k}\omega_{i}\big)\nonumber\\
&\quad{}+(m+1)((1-\frac{p}{2})m+n-m-3)E_{n-4}+(m+1)\omega_{m}E_{n-m-3}\nonumber\\
&\quad{}+\sum_{\begin{subarray}{l}0\leq i\leq m,m+1\leq j\nonumber\\(i,j)\neq (m+1,0),\\k=n-i-j-2\end{subarray}}
E_{k}(j-i)\omega_{i+j-1}
+\sum_{\begin{subarray}{l}0\leq i\leq m,m+1\leq j,\nonumber\\k=n-i-j-2\end{subarray}}((1-\frac{p}{2})i+k)E_{i+k-1}\omega_{j}+\omega_{i}E_{k}\omega_{j}\big)\nonumber\\
&\quad{}+\sum_{\begin{subarray}{l}m+1\leq i,j,\\k=n-i-j-2\end{subarray}}E_{k}\omega_{j}\omega_{i}.
\end{align}
For $i\in\Z$ and a subset $X$ of a weak $M(1)^{+}$-module $\mK$,
$\langle\omega_i\rangle X$ denotes the subspace of $\mK$ spanned by the elements
$\omega_i^{j}\lu, j\in\Z_{\geq 0}, \lu\in X$. 

The following lemmas follow from \eqref{eq:wiwj}--\eqref{eq:wiEji}, \eqref{eq:[omegai(Har0ExB)j]},
and \eqref{eq:[omegaiHar1ExBj]}.

\begin{lemma}\label{lemma:ojzero}
Let $\lu$ be  
an element of a weak $M(1)^{+}$-module $(\mK,Y_{\mK})$ with $\epsilon(\omega,\lu)=\epsilon_{Y_{\mK}}(\omega,\lu)\geq 1$.
\begin{enumerate}
\item For any $i\geq 0$, $j>\epsilon(\omega,\lu)$, and $k\geq 2$, 
\begin{align}
\omega_{j}\omega_{\epsilon(\omega,\lu)}^{i}\lu&=
\omega_{j}\omega_{\epsilon(\omega,\lu)}^{i}J_{\epsilon(J,\lu)}\lu=0,\nonumber\\
\omega_{k}^{i}J_{\epsilon(J,\lu)}\lu&=J_{\epsilon(J,\lu)}\omega_{k}^{i}\lu.
\end{align}
\item
For any $\lv\in \langle\omega_{\epsilon(\omega,\lu)}\rangle\{\lu,J_{\epsilon(J,\lu)}\lu\}$,
$\epsilon(\omega,\lv)\leq \epsilon(\omega,\lu)$. 
\end{enumerate}
\end{lemma}

\begin{lemma}\label{lemma:EBasis-zero}
Let $\lu\in\module$ with $\epsilon(\omega,\lu)=\epsilon_{I}(\omega,\lu)\geq 1$
and let $a$ be one of $E,\Har_{0}E$, or $\Har_{1}E$.
\begin{enumerate}
\item For any $i\geq 0$, $j>\epsilon(\omega,\lu)$, and $k\geq 2$, 
\begin{align}
\omega_{j}\omega_{\epsilon(\omega,\lu)}^{i}a_{\epsilon(a,\lu)}\lu&=0,\nonumber\\
\omega_{k}^{i}a_{\epsilon(a,\lu)}\lu&=a_{\epsilon(a,\lu)}\omega_{k}^{i}\lu.
\end{align}
\item
For any $\lv\in \langle\omega_{\epsilon(\omega,\lu)}\rangle a_{\epsilon(a,\lu)}\lu$,
$\epsilon(\omega,\lv)\leq\epsilon(\omega,\lu)$. 
\end{enumerate}
\end{lemma}

By using the commutation relation $[\Har_{i},\ExB_{j}]=\sum_{k=0}^{\infty}\binom{i}{k}(\Har_{k}\ExB)_{i+j-k}$
for $i,j\in\Z$,
the following result follows from
Lemma \ref{lemma:comm-change} and \eqref{eq:Har1ExB=dfrac2mn2mn1omega1ExB}--\eqref{eq:J1ExB=frac2mn(4mn-11)-2}.
\begin{lemma}\label{lemma:bound-H0E}
Assume $\mn\neq 1/2$.
Let $\lu\in \mW$ with $\epsilon(\omega,\lu)\geq 1$
and $\epsilon(\Har,\lu)\leq 2\epsilon(\omega,\lu)+1$.
Then, $\epsilon(\Har_{0}\ExB,\lu)\leq \epsilon(\ExB,\lu)+2\epsilon(\omega,\lu)+1$.
\end{lemma}

If $\mn=1/2$, then a direct computation shows that
\begin{align}
0&=\omega_{-1 } E-\omega_{0 }^2 E,\label{eq:norm1-2-0}\\
\Har_{0}E&=\frac{-2}{3}\omega_{-2 } E+\frac{4}{3}\omega_{0 }(\Har_{1}E),\label{eq:norm1-2-1}\\
0&=
8\omega_{-3 } E+12\Har_{-1 } E
+3\omega_{-1 } (\Har_{1}E)+4\omega_{0 } \omega_{-2 } E
-11\omega_{0 }^2(\Har_{1}E).
\label{eq:norm1-2-2}
\end{align}
\begin{lemma}
\label{lemma:bound-H0-1E}
Let $U$ be an $A(M(1)^{+})$-submodule of $\Omega_{M(1)^{+}}(\mW)$,
$\lu\in U$, and $\lE\in\Z$ such that $\epsilon(\ExB,\lv)\leq \lE$ for all $\lv\in U$. 
Then $\epsilon(\Har_{0}\ExB,\lu)\leq \lE+3$ and 
$\epsilon(\Har_{1}\ExB,\lu)\leq \lE+2$.
\end{lemma}
\begin{proof}
For $\mn\neq 1/2$,
the result follows from
Lemma \ref{lemma:bound-H0E} and \eqref{eq:J1ExB=frac2mn(4mn-11)-1}.
Assume $\mn=1/2$.
For $i,j\in\Z$ and $r\in\Z_{\geq 0}$, it follows from 
\eqref{eq:Har1ExB=dfrac2mn2mn1omega1ExB}, \eqref{eq:norm1-2-1}, and Lemma \ref{lemma:comm-change} that
\begin{align}
\label{eq:[HariExBj]lu}
&[\Har_{i},\ExB_{j}]\lu\nonumber\\
&=(\Har_{0}\ExB)_{i+j}\lu+i(\Har_{1}\ExB)_{i+j-1}\lu+\binom{i}{2}(\Har_{2}\ExB)_{i+j-2}\lu\nonumber\\
&=\frac{-1}{3}(i+4j)(\Har_{1}E)_{i+j-1}\lu
-\binom{i}{2}\frac{i+j-2}{3}\ExB_{i+j-3}\lu\nonumber\\
&\quad{}+\frac{-2}{3}
\big(\sum_{k\leq \lom}(-k-1)\omega_{k}\ExB_{i+j-2-k}+\sum_{k\geq \lom+1}(-k-1)\ExB_{i+j-2-k}\omega_{k}\nonumber\\
&\qquad{}+(\binom{r+2}{2}(-i-j+2)+\binom{r+2}{3})E_{i+j-3}\big)\lu.
\end{align}
Using \eqref{eq:[HariExBj]lu} with $i=3$ and $r=1$,
we have $\epsilon(\Har_{1}\ExB,\lu)\leq \lE+2$.
By \eqref{eq:norm1-2-1}, $\epsilon(\Har_{0}\ExB,\lu)\leq \lE+3$.
\end{proof}

A direct computations shows the following result.
\begin{lemma}
\label{lemma:relations-M(1)-V(lattice)+}
The following elements of $V_{\lattice}^{+}$ are zero:
\begin{align}
\sv^{(8),H}&=
-2376\omega_{-2 } \omega_{-2 } \omega_{-1 } \vac
+3168\omega_{-3 } \omega_{-1 } \omega_{-1 } \vac
-6256\omega_{-3 } \omega_{-3 } \vac
-11799\omega_{-4 } \omega_{-2 } \vac
\nonumber\\&\quad{}
+30456\omega_{-5 } \omega_{-1 } \vac
+2310\omega_{-7 } \vac
-9504\omega_{-1 } \omega_{-1 } \Har_{-1 } \vac
-6024\omega_{-3 } \Har_{-1 } \vac
\nonumber\\&\quad{}
-13419\omega_{-2 } \Har_{-2 } \vac
-6516\omega_{-1 } \Har_{-3 } \vac
+11868\Har_{-5 } \vac+5040\Har_{-1 }^2 \vac,\\
\sv^{(8),J}&=
-29056\omega_{-1}^4\vac
-118960\omega_{-2}^2\omega_{-1}\vac+
39040\omega_{-3}\omega_{-1}^2\vac
-39480\omega_{-3}^2\vac\nonumber\\
&\quad{}
-32120\omega_{-4}\omega_{-2}\vac+
497760\omega_{-5}\omega_{-1}\vac+
230360\omega_{-7}\vac\nonumber\\
&\quad{}+
5024\omega_{-1}^2J_{-1}\vac
-8536\omega_{-3}J_{-1}\vac+
8939\omega_{-2}J_{-2}\vac\nonumber\\
&\quad{}
-2444\omega_{-1}J_{-3}\vac+
1572J_{-5}+560J_{-1}^2\vac,\\
\sv^{(9)}&=30J_{-6}\vac-30\omega_{-1}J_{-4}\vac+27\omega_{-2}J_{-3}\vac-39\omega_{-3}J_{-2}\vac\nonumber\\
&\quad{}+16\omega_{-1}^2J_{-2}\vac+52\omega_{-4}J_{-1}\vac-32\omega_{-2}\omega_{-1}J_{-1}\vac,
\end{align}
\begin{align}
\sv^{(10),H}&=
919328\omega_{-9 } \vac
-545856\omega_{-5 } \omega_{-1 } \omega_{-1 } \vac
\nonumber\\&\quad{}
-529536\omega_{-4 } \omega_{-4 } \vac
+545352\omega_{-4 } \omega_{-2 } \omega_{-1 } \vac
\nonumber\\&\quad{}
+520160\omega_{-3 } \omega_{-3 } \omega_{-1 } \vac
-524968\omega_{-3 } \omega_{-2 } \omega_{-2 } \vac
\nonumber\\&\quad{}
-10240\omega_{-3 } \omega_{-1 } \omega_{-1 } \omega_{-1 } \vac
+7680\omega_{-2 } \omega_{-2 } \omega_{-1 } \omega_{-1 } \vac
\nonumber\\&\quad{}
+1937712\omega_{-5 } \Har_{-1 } \vac
-845376\omega_{-3 } \omega_{-1 } \Har_{-1 } \vac
\nonumber\\&\quad{}
-381048\omega_{-2 } \omega_{-2 } \Har_{-1 } \vac
+30720\omega_{-1 } \omega_{-1 } \omega_{-1 } \Har_{-1 } \vac
\nonumber\\&\quad{}
-720081\omega_{-4 } \Har_{-2 } \vac
-128280\omega_{-2 } \omega_{-1 } \Har_{-2 } \vac
\nonumber\\&\quad{}
-435576\omega_{-3 } \Har_{-3 } \vac
+234528\omega_{-1 } \omega_{-1 } \Har_{-3 } \vac
\nonumber\\&\quad{}
+345849\omega_{-2 } \Har_{-4 } \vac
-1211160\omega_{-1 } \Har_{-5 } \vac
\nonumber\\&\quad{}
+2360970\Har_{-7 } \vac
+70875\Har_{-2 } \Har_{-2 } \vac
\nonumber\\&\quad{}
+734184\omega_{-7 } \omega_{-1 } \vac
+898766\omega_{-6 } \omega_{-2 } \vac,
\end{align}
\begin{align}
\sv^{(10),J}&=8192\omega_{-1 }^5\vac-2048\omega_{-1 }^{3}J_{-1 } \vac\nonumber\\
&\quad{}+758496\omega_{-9 } \vac
-1728\omega_{-5 } \omega_{-3 } \vac
\nonumber\\&\quad{}
-15232\omega_{-5 } \omega_{-1 } \omega_{-1 } \vac
-60848\omega_{-4 } \omega_{-4 } \vac
\nonumber\\&\quad{}
-134224\omega_{-4 } \omega_{-2 } \omega_{-1 } \vac
-6912\omega_{-3 } \omega_{-3 } \omega_{-1 } \vac
\nonumber\\&\quad{}
-136872\omega_{-3 } \omega_{-2 } \omega_{-2 } \vac
-112640\omega_{-3 } \omega_{-1 } \omega_{-1 } \omega_{-1 } \vac
\nonumber\\&\quad{}
-69280\omega_{-2 } \omega_{-2 } \omega_{-1 } \omega_{-1 } \vac
-6092\omega_{-4 } J_{-2 } \vac
\nonumber\\&\quad{}
+6272\omega_{-3 } \omega_{-1 } J_{-1 } \vac
+360\omega_{-2 } \omega_{-2 } J_{-1 } \vac
\nonumber\\&\quad{}
+152\omega_{-2 } \omega_{-1 } J_{-2 } \vac
+1856\omega_{-3 } J_{-3 } \vac
\nonumber\\&\quad{}
+9408\omega_{-1 } \omega_{-1 } J_{-3 } \vac
+12656\omega_{-2 } J_{-4 } \vac
\nonumber\\&\quad{}
-29968\omega_{-1 } J_{-5 } \vac
+43320J_{-7 } \vac
\nonumber\\&\quad{}
+525J_{-2 } J_{-2 } \vac
+1309248\omega_{-7 } \omega_{-1 } \vac
\nonumber\\&\quad{}
+352992\omega_{-6 } \omega_{-2 } \vac,
\end{align}
\begin{align}
\label{eq:big(2(mn-2)(-27 + 54mn - 44mn2+ 40mn3)omega-3-1}
Q^{(4)}
&=
2(\mn-2)(-27 + 54 \mn - 44 \mn^2 + 40 \mn^3)\omega_{-3 }\ExB
\nonumber\\&\quad{}
-12\mn (\mn-2) (-3 + 4 \mn)\omega_{-1 }^2\ExB
\nonumber\\&\quad{}
-6 \mn( \mn-2 )  (-9 + 2 \mn) (-1 + 2 \mn)\Har_{-1 }\ExB
\nonumber\\&\quad{}
+(-72\mn^3-96\mn^2+210\mn-90)\omega_{0 } \omega_{-2 }\ExB
\nonumber\\&\quad{}
+(120\mn^2-48\mn+36)\omega_{0 }^2\omega_{-1 }\ExB
\nonumber\\&\quad{}
+(-48\mn-9)\omega_{0 }^4\ExB,
\end{align}
\begin{align}
Q^{(5,1)}&=
3 (\mn-2) (10 \mn^2-29 \mn+32) (10 \mn^2-4 \mn+3)\omega_{-4 } E
\nonumber\\&\quad{}
-12 \mn (3 \mn-4) (10 \mn^2-4 \mn+3)\omega_{-2 } \omega_{-1 } E
\nonumber\\&\quad{}
-3 (\mn-8) (\mn-2) (2 \mn-1) (10 \mn^2-4 \mn+3)\Har_{-2 } E
\nonumber\\&\quad{}
+8 (2 \mn-7) (15 \mn^3-22 \mn^2+8 \mn-6)\omega_{0 } \omega_{-3 } E
\nonumber\\&\quad{}
+24 \mn^2 (8 \mn-9)\omega_{0 } \omega_{-1 } \omega_{-1 } E
\nonumber\\&\quad{}
-12 (\mn-2) (2 \mn-1) (6 \mn^2-5 \mn+6)\omega_{0 } \Har_{-1 } E
\nonumber\\&\quad{}
-6 (2 \mn^3-32 \mn^2+29 \mn+12)\omega_{0 } \omega_{0 } \omega_{-2 } E
\nonumber\\&\quad{}
-6 (8 \mn-9)\omega_{0 } \omega_{0 } \omega_{0 } \omega_{0 } \omega_{0 } E,\\
Q^{(5,2)}&=
3(\mn-2)(10\mn^2-29\mn+32)(12\mn^3+16\mn^2-35\mn+15)\omega_{-4 } \ExB
\nonumber\\&\quad{}
-12\mn(3\mn-4)(12\mn^3+16\mn^2-35\mn+15)\omega_{-2 } \omega_{-1 } \ExB
\nonumber\\&\quad{}
-3(\mn-8)(\mn-2)(2\mn-1)(12\mn^3+16\mn^2-35\mn+15)\Har_{-2 } \ExB
\nonumber\\&\quad{}
+2(136\mn^5-316\mn^4-1266\mn^3+3409\mn^2-2470\mn+624)\omega_{0 } \omega_{-3 } \ExB
\nonumber\\&\quad{}
+12\mn(20\mn^3-3\mn^2-44\mn+24)\omega_{0 } \omega_{-1 } \omega_{-1 } \ExB
\nonumber\\&\quad{}
-6(\mn-2)(2\mn-1)(14\mn^3+21\mn^2-74\mn+60)\omega_{0 } \Har_{-1 } \ExB
\nonumber\\&\quad{}
-12(2\mn^3-32\mn^2+29\mn+12)\omega_{0 } \omega_{0 } \omega_{0 } \omega_{-1 } \ExB
\nonumber\\&\quad{}
-3(16\mn^2+61\mn-102)\omega_{0 } \omega_{0 } \omega_{0 } \omega_{0 } \omega_{0 } \ExB,
\end{align}
\begin{align}
Q^{(6)}&=
 2 (3696 \mn^8-22564 \mn^7+66284 \mn^6-84937 \mn^5+56207 \mn^4
\nonumber\\&\qquad{}
-91528 \mn^3+11774 \mn^2+29190 \mn-13500)\omega_{-5 } E
\nonumber\\&\quad{}
-4 \mn (352 \mn^6+2152 \mn^5-8282 \mn^4+7951 \mn^3-11696 \mn^2\nonumber\\&\qquad{}
+6304 \mn-1542)\omega_{-3 } \omega_{-1 } E
\nonumber\\&\quad{}
-3 \mn (1584 \mn^6-5572 \mn^5+6456 \mn^4-6877 \mn^3+5214 \mn^2\nonumber\\&\qquad{}
-3040 \mn+642)\omega_{-2 } \omega_{-2 } E
\nonumber\\&\quad{}
+720 \mn^3 (\mn-2) (4 \mn-1)\omega_{-1 } \omega_{-1 } \omega_{-1 } E
\nonumber\\&\quad{}
-24 \mn (\mn-2) (2 \mn-1) (44 \mn^4-98 \mn^3+157 \mn^2-88 \mn+48)\omega_{-1 } \Har_{-1 } E
\nonumber\\&\quad{}
-3 (\mn-2) (2 \mn-25) (2 \mn-1)^2 (44 \mn^4-13 \mn^3+62 \mn^2-48 \mn+18)\Har_{-3 } E
\nonumber\\&\quad{}
+3 (1760 \mn^7-9382 \mn^6+1391 \mn^5+28130 \mn^4-14380 \mn^3\nonumber\\&\qquad{}
+29762 \mn^2-25851 \mn+7650)\omega_{0 } \omega_{-4 } E
\nonumber\\&\quad{}
+12 \mn (352 \mn^5-1459 \mn^4+2396 \mn^3-2894 \mn^2+1254 \mn-225)\omega_{0 } \omega_{-2 } \omega_{-1 } E
\nonumber\\&\quad{}
-3 (\mn-2) (2 \mn-1) (352 \mn^5+101 \mn^4+86 \mn^3-614 \mn^2+804 \mn-225)\omega_{0 } \Har_{-2 } E
\nonumber\\&\quad{}
+12 (88 \mn^6+1104 \mn^5-4136 \mn^4+3714 \mn^3-3944 \mn^2+2670 \mn-675)\omega_{0 } \omega_{0 } \omega_{-3 } E
\nonumber\\&\quad{}
-6 (352 \mn^5-1099 \mn^4+686 \mn^3-689 \mn^2+804 \mn-225)\omega_{0 } \omega_{0 } \omega_{0 } \omega_{-2 } E
\nonumber\\&\quad{}
-90 (\mn-2) (4 \mn-1)\omega_{0 } \omega_{0 } \omega_{0 } \omega_{0 } \omega_{0 } \omega_{0 } E.
\end{align}
If $\mn=2$, then we have the four following relations:
\begin{align}
0&=6\omega_{-2 } E
-4\omega_{0 } \omega_{-1 } E
+\omega_{0 }^{3}E,\label{eq:6omega2E4omega0omega1E-1}\\
0&=
180\omega_{-3 }E  -48\omega_{-1 }^2 E +72H_{-1 }E -63\omega_{0 } (H_{0}E)\nonumber\\&\quad{}
 +8\omega_{0 }^2\omega_{-1 } E +\omega_{0 }^4 E,
\label{eq:6omega2E4omega0omega1E-2}
\\
0&=
9450\omega_{-4 } E -900\omega_{-1 } (H_{0}E)+6750H_{-2 } E-768\omega_{0 } \omega_{-1 }^{2} E \nonumber\\&\quad{}
 -3168\omega_{0 } \Har_{-1 } E 
+297\omega_{0 }^{2} (H_{0}E) +128\omega_{0 }^{3}\omega_{-1 } E+16\omega_{0 }^{5}E,
\label{eq:6omega2E4omega0omega1E-3}\\
0&=
584199000\omega_{-6 } \ExB 
-117085500\Har_{-4 } \ExB 
\nonumber\\&\quad{}
+98941500\omega_{-3 } (\Har_{0}\ExB) 
-27594000\omega_{-1 }^2 (\Har_{0}\ExB) 
\nonumber\\&\quad{}
+34587000\Har_{-1 } (\Har_{0}\ExB) 
-13132800\omega_{0 } \omega_{-1 }^3 \ExB 
\nonumber\\&\quad{}
-60739200\omega_{0 } \omega_{-1 } \Har_{-1 } \ExB 
+277223400\omega_{0 } \Har_{-3 } \ExB 
\nonumber\\&\quad{}
-85188900\omega_{0 } \omega_{-2 } (\Har_{0}\ExB) 
+206053320\omega_{0 }^2 \omega_{-4 } \ExB 
\nonumber\\&\quad{}
-8524040\omega_{0 }^2 \omega_{-1 } (\Har_{0}\ExB) 
-27546608\omega_{0 }^3\omega_{-3 } \ExB 
\nonumber\\&\quad{}
-51990312\omega_{0 }^3\Har_{-1 } \ExB 
+17161013\omega_{0 }^4(\Har_{0}\ExB) 
\nonumber\\&\quad{}
-820800\omega_{0 }^5\omega_{-1 } \ExB 
+410400\omega_{0 }^7 \ExB.
\label{eq:6omega2E4omega0omega1E-4}
\end{align}
\renewcommand{\arraystretch}{2.0}
\renewcommand{\arraystretch}{1}
\end{lemma}

\begin{lemma}\label{lemma:m1+wJ}
Let $\lu$ be a non-zero element of a weak $M(1)^{+}$-module $(\mK,Y_{\mK})$ such that
$\epsilon(\omega,\lu)=\epsilon_{Y_{\mK}}(\omega,\lu)\geq 2$ and  $\epsilon(\omega,\lu)\leq \epsilon(\omega,\lv)$ for 
all non-zero $\lv\in \mK$.
Then $\epsilon(J,\lu)=2\epsilon(\omega,\lu)+1$,
\begin{align}
\label{eq:J2e=4omegae}
J_{2\epsilon(\omega,\lu)+1}\lu&=4\omega_{\epsilon(\omega,\lu)}^2\lu,
\end{align}
and 
\begin{align}
\label{eq:epsilon(Har,lu)leq 2epsilon}
\epsilon(\Har,\lu)&\leq 2\epsilon(\omega,\lu).
\end{align} 
\end{lemma}
\begin{proof}
We write 
\begin{align}
\lao&=\epsilon(\omega,\lu)\mbox{ and }\laJ=\epsilon(J,\lu)
\end{align}
for simplicity. We note that $\laJ\in\Z$ by \eqref{eq:epsilonI(u,v)inZ}.
It follows from Lemma \ref{lemma:ojzero} (2) and the condition of $\lu$ that for any non-zero $\lv\in \langle\omega_{\lao}
\rangle\{\lu,J_{\laJ}\lu\}$ and 
$i\in\Z_{\geq 0}$, 
\begin{align}
\label{eq:omiv}
\omega_{\lao}^i\lv&\neq 0.
\end{align}
Since the same argument as in \cite[(3.23)]{Tanabe2017}
shows that
\begin{align}
\label{eq:J-ind}
0&=\dfrac{1}{16}\sv^{(9)}_{\laJ+2\lao+3}\lu=(-\laJ+2\lao+1)J_{\laJ}\omega_{\lao}^2\lu\nonumber\\
&=(-\laJ+2\lao+1)\omega_{\lao}^2J_{\laJ}\lu
\end{align}
by Lemma \ref{lemma:ojzero} (1), $\laJ=2\lao+1$ by \eqref{eq:omiv}.
Since
\begin{align}
0&=P^{(10),J}_{5\lao+4}\lu
=(8192\omega_{-1}^{5}\vac-2048\omega_{-1}^3J_{-1}\vac)_{5\lao+4}\lu\nonumber\\
&=2048(4\omega_{\lao}^{5}-J_{2\lao+1}\omega_{\lao}^3)\lu\nonumber\\
&=2048\omega_{\lao}^{3}(4\omega_{\lao}^{2}-J_{2\lao+1})\lu
\end{align}
by Lemma \ref{lemma:ojzero} (1), \eqref{eq:J2e=4omegae} holds by \eqref{eq:omiv}.
It follows from \eqref{eq:definition-omega-J-H} that $\Har_{i}\lu=0$ for all $i\geq 2\epsilon(\omega,\lu)+1$ and hence 
$\epsilon(\Har,\lu)\leq 2\epsilon(\omega,\lu)$.
\end{proof}

\begin{lemma}\label{lemma:r=1-s=3}
	Let $\lattice$ be a non-degenerate even lattice of rank $1$
	and $\module$ a non-zero weak $V_{\lattice}^{+}$-module.
	Then, there exists a non-zero $\lu \in\Omega_{M(1)^{+}}(\module)$ that
	satisfies one of the following conditions:
	\begin{enumerate}
		\item $\epsilon(\omega,\lu)=\epsilon(J,\lu)=\epsilon(E,\lu)=-1$. In this case $V_{L}^{+}\cdot \lu\cong V_{L}^{+}$.
		\item $\Har_{3}\lu=0$.
		\item $\omega_1\lu=\lu$ and $\Har_{3}\lu=\lu$.
		\item $\omega_1\lu=(1/16)\lu$ and $\Har_{3}\lu=(-1/128)\lu.$
		\item $\omega_1\lu=(9/16)\lu$ and $\Har_{3}\lu=(15/128)\lu.$
	\end{enumerate}
\end{lemma}
\begin{proof}
We write $\lattice=\Z \alpha$. Throughout the proof of this lemma,
$\mn=\langle\alpha,\alpha\rangle\in 2\Z\setminus\{0\}$.
	For a non-zero $\lu\in \mW$ with $\epsilon(\omega,\lu)<0$,
	since $\omega_0\lu=0$, it follows from \cite[Proposition 4.7.7]{LL}
	that $V_{\lattice}^{+}\cdot \lu\cong V_{\lattice}^{+}$ and hence
	$\epsilon(\omega,\lu)=\epsilon(J,\lu)=\epsilon(E,\lu)=-1$.
	
	We assume $\epsilon(\omega,\lv)\geq 0$ for all non-zero $\lv\in\module$.
	We take a non-zero $\lu\in\module$ with  $\epsilon(\omega,\lu)$ as small as possible, namely $0\leq \epsilon(\omega,\lu)\leq \varepsilon(\omega,\lv)$ for all
non-zero $\lv\in \module$.
	We write 
	\begin{align}
		\lom&=\epsilon(\omega,\lu),\quad \lJ=\epsilon(J,\lu),\mbox{ and }\lE=\epsilon(E,\lu)
	\end{align}
	for simplicity. We note that $\laJ,\lE\in\Z$ by \eqref{eq:epsilonI(u,v)inZ}.
	Suppose $\lao\geq 2$. Then, Lemma \ref{lemma:m1+wJ}  shows that $\laJ=2\lao+1$ and
	$
	\Har_{i}\lu=0
	$
	for all $i\geq 2\lao+1$.
By Lemma \ref{lemma:EBasis-zero} (2), 
for any non-zero $\lv\in \langle\omega_{\lao}\rangle\{\lu,J_{\laJ}\lu,E_{\laE}\lu, 
(\Har_{0}E)_{\epsilon(\Har_{0}E,\lu)}\lu\}$ and 
	$i\in\Z_{\geq 0}$, 
	\begin{align}
		\label{eq:omiv2}
		\omega_{\lao}^i\lv&\neq 0.
	\end{align}
Assume $\mn\neq 2$.
By \eqref{eq:(omega-12E)wn-1} and \eqref{eq:(omega-12E)wn-2} with $m=\lom$, 
	\begin{align}
		\label{eq:0=(omega-3ExB)lE+2lom+2lu=}
0&=(\omega_{-3}\ExB)_{\lE+2\lom+2}\lu=(\omega_0\omega_{-2}\ExB)_{\lE+2\lom+2}\lu\nonumber\\
&=(\omega_0^2\omega_{-1}\ExB)_{\lE+2\lom+2}\lu=
		(\omega_0^4\ExB)_{\lE+2\lom+2}\lu
		\end{align}
	and
	\begin{align}
	\label{eq:(omega-12ExB)lE+2lom+2lu=}
	(\omega_{-1}^2\ExB)_{\lE+2\lom+2}\lu&=\omega_{\lom}^{2}\ExB_{\lE}\lu.
	\end{align}
Using Lemma \ref{lemma:comm-change}, 
\eqref{eq:Har1ExB=dfrac2mn2mn1omega1ExB}, \eqref{eq:J1ExB=frac2mn(4mn-11)-1}, and \eqref{eq:J1ExB=frac2mn(4mn-11)-2},
we expand $(\Har_{-1}\ExB)_{\lE+2\lom+2}$ so that 
the resulting expression is  a linear combination of elements of the form 
\begin{align}
a^{(1)}_{i_1}\cdots a^{(l)}_{i_l}\ExB_{m}b^{(1)}_{j_1}\cdots b^{(n)}_{j_n}
	\end{align}
where $l,n\in\Z_{\geq 0}$, $m\in\Z$, and 
\begin{align}
(a^{(1)},i_1),\ldots,(a^{(l)},i_l)&\in\{(\omega,k)\ |\ k\leq \lom\}\cup\{(\Har,k)\ |\ k\leq 2\lom\},\nonumber\\
(b^{(1)},j_1),\ldots,(b^{(n)},j_n)&\in\{(\omega,k)\ |\ k\geq \lom+1\}\cup\{(\Har,k)\ |\ k\geq 2\lom+1\},
\end{align}
as was done in \eqref{eq:(omega-12E)wn-1} and \eqref{eq:(omega-12E)wn-2}.
Then, taking the action of the obtained expansion of $(\Har_{-1}\ExB)_{\lE+2\lom+2}$ on $\lu$
and using \eqref{eq:epsilon(Har,lu)leq 2epsilon} and \eqref{eq:0=(omega-3ExB)lE+2lom+2lu=}, we have 
	\begin{align}
	\label{eq:(Har-1ExB)lE+2lom+2lu=}
	(\Har_{-1}\ExB)_{\lE+2\lom+2}\lu&=\ExB_{\lE}\Har_{2\lom+1}\lu=0.
\end{align}
By \eqref{eq:big(2(mn-2)(-27 + 54mn - 44mn2+ 40mn3)omega-3-1}, \eqref{eq:0=(omega-3ExB)lE+2lom+2lu=}, 
\eqref{eq:(omega-12ExB)lE+2lom+2lu=}, and \eqref{eq:(Har-1ExB)lE+2lom+2lu=},
	\begin{align}
		\label{eqn:zeroQ4}
		0&=Q^{(4)}_{\lE+2\lom+2}\lu
		=-12\mn (\mn-2) (-3 + 4 \mn)\omega_{\lom}^2E_{\lE}\lu,
	\end{align}
which contradicts \eqref{eq:omiv2}. 

Assume $\mn=2$.
By Lemma \ref{lemma:bound-H0E}, $\epsilon(\Har_{0}\ExB,\lu)\leq \lE+2\lom+1$.
	By \eqref{eq:6omega2E4omega0omega1E-2}, Lemma \ref{lemma:m1+wJ} and the results in Section \ref{section:normal-2},
the same argument as above shows
	\begin{align}
		0&=
		(180\omega_{-3 } E -48\omega_{-1 }^2 E +72\Har_{-1 } E -63\omega_{0 } (\Har_{0}E)\nonumber\\&\quad{}
		+8\omega_{0 }^2\omega_{-1 } E +\omega_{0 }^4 E)_{\lE+2\lom+2}\lu\nonumber\\
		&=(-48\omega_{\lom}^2E_{\lE}+72E_{\lE}\Har_{2\lom+1}+63(\lE+2\lom+2)(\Har_{0}E)_{\lE+2\lom+1})\lu\nonumber\\
		&=(-48\omega_{\lom}^2E_{\lE}+63(\lE+2\lom+2)(\Har_{0}E)_{\lE+2\lom+1})\lu
	\end{align}
	and hence $(\Har_{0}E)_{\lE+2\lom+1}\lu\neq 0$ by \eqref{eq:omiv2}.
	By \eqref{eq:6omega2E4omega0omega1E-3} and results in Section \ref{section:normal-2},
	\begin{align}
		0&=
		(9450\omega_{-4 } E -900\omega_{-1 } (H_{0}E)+6750H_{-2 } E-768\omega_{0 } \omega_{-1 }^{2} E \nonumber\\&\quad{}
		-3168\omega_{0 } \Har_{-1 } E 
		+297\omega_{0 }^{2} (H_{0}E) +128\omega_{0 }^{3}\omega_{-1 } E+16\omega_{0 }^{5}E)_{\lE+3\lom+2}\lu\nonumber\\
		&=-900\omega_{\lom}(\Har_{0}E)_{\lE+2\lom+1}\lu,
	\end{align}
	which also contradicts \eqref{eq:omiv2}. We conclude that $\lao\leq 1$.
	
	Suppose $\laJ\geq 4$. 
By using  \cite[(2.29)]{Tanabe2017} and \eqref{eq:wiJj}, the same argument as in \cite[Lemma 3.3]{Tanabe2017} shows that
	$\epsilon(\omega,J_{\laJ}\lu)\leq 1$ and
	$J_{j}J_{\laJ}\lu=0$ for all $j\geq \laJ+1$.
	By the same argument as in \cite[(3.25)]{Tanabe2017},
	\begin{align}
		(J_{-1}J)_{2\laJ+1}\lu
		&=J_{\laJ}^2\lu
	\end{align}
	and hence
	\begin{align}
		\label{eq:jjv}
		0&=P^{(8),J}_{2\laJ+1}\lu=J_{\laJ}^2\lu=J_{\laJ}(J_{\laJ}\lu),
	\end{align}
	which means $\epsilon(J,J_{\laJ}\lu)<\laJ=\epsilon(J,\lu)$.
	Replacing $\lu$ by $J_{\laJ}\lu$ repeatedly, we get a non-zero $\lu\in \module$ such that
	$\lao\leq 1$ and $\laJ\leq 3$. Thus, $\lu\in \Omega_{M(1)^{+}}(\module)$ and 
	in particular, $\epsilon(\Har,\lu)\leq 3$.
	Deleting the terms including $\omega_1^{i}\Har_3^2\lu\ (i=0,1,\ldots)$ from 
	the following simultaneous equations 
	\begin{align}
		\label{eq:p87u}
		0&=P^{(8),H}_7\lu
		=-72(132 \omega_{1 }^2
		-65 \omega_{1 } 
		+3
		-70\Har_{3 })\Har_{3 }\lu\mbox{ and }\\
		0&=P^{(10),H}_9\lu \nonumber\\
		&=240 \Har_{3} (-207 + 4725 \Har_{3} + 4472 \omega_{1} - 9118 \omega_{1}^2 + 128 \omega_{1}^3)\lu,
	\end{align}
	we have
	\begin{align}
		\label{eq:P8-P10}
		0
		&=( \omega_{1}-1 ) ( 16 \omega_{1}-1 ) (16 \omega_{1}-9  ) \Har_{3}\lu.
	\end{align}
	By \eqref{eq:p87u} and \eqref{eq:P8-P10}, the proof is complete.
\end{proof}
\begin{remark}
	If $\mn>0$, then Lemma \ref{lemma:r=1-s=3} also 
	follows from \cite[Theorem 7.7]{Abe2005}, \cite[Theorem 5.13]{DN1999-2}, and \cite[Theorem 2.7]{Mi2004d}.
\end{remark}
\begin{remark}
As we have seen in the proof of Lemma \ref{lemma:r=1-s=3},
starting from an arbitrary non-zero element in $\module$,
we can get $\lu$ in Lemma \ref{lemma:r=1-s=3} inductively.
\end{remark}

\begin{lemma}
\label{lemma:structure-Vlattice-M1}
Assume $\mn\neq 2, 1/2$.
Let $\lu$ be a non-zero element of $\Omega_{M(1)^{+}}(W)$ with $I(E,x)\lu\neq 0$.
We write 
\begin{align}
\lE&=\epsilon(\ExB,\lu)
\end{align}
for simplicity.
We set
\begin{align}
\lv&=(\omega_1-\dfrac{(\lE+1)^2}{2\mn})\lu.
\end{align}
We have 
\begin{align}
\lvE&=(\omega_1-\dfrac{(\lE+1-\mn)^2}{2\mn})E_{\lE}\lu.
\end{align}
\begin{enumerate}
\item Assume $\Har_{3}\lu=0$.
If $\lvE\neq 0$, then 
$\lE=\mn-2$ and 
\begin{align}
\label{eq:omega1lv=lv}
\omega_1(\lvE)&=\lvE.
\end{align}
If $\lu$ is an eigenvector of $\omega_1$ and $\lv\neq 0$, then 
$\lE=\mn-2$ and 
\begin{align}
\omega_1\lu=\dfrac{\mn}{2}\lu.
\end{align}
\item
If $\omega_1\lu=\lu$ and $\Har_{3}\lu=\lu$, then
$\lE=0$.
\item
Assume $\omega_1\lu=(1/16)\lu$ and $\Har_{3}\lu=(-1/128)\lu$. 
Then
$\lE=\mn/2-1$ or $(\mn-1)/2$.
In particular if $\mn$ is an even integer, then $\lE=\mn/2-1$.
\item
Assume $\omega_1\lu=(9/16)\lu$ and $\Har_{3}\lu=(15/128)\lu$. Then
$\lE=\mn/2-1$ or $(\mn-3)/2$.
In particular if $\mn$ is an even integer, then $\lE=\mn/2-1$.
\end{enumerate}
\end{lemma}
\begin{proof}
We first expand each of $Q^{(4)}_{\lE+4},Q^{(5,1)}_{\lE+5}, Q^{(5,2)}_{\lE+5}$, and $Q^{(6)}_{\lE+6}$
so that the resulting expression is a linear combination of elements of the form 
\begin{align}
	a^{(1)}_{i_1}\cdots a^{(l)}_{i_l}\ExB_{m}b^{(1)}_{j_1}\cdots b^{(n)}_{j_n}
\end{align}
where $l,n\in\Z_{\geq 0}$, $m\in\Z$, and 
\begin{align}
	(a^{(1)},i_1),\ldots,(a^{(l)},i_l)&\in\{(\omega,k)\ |\ k\leq 1\}\cup\{(\Har,k)\ |\ k\leq 2\},\nonumber\\
	(b^{(1)},j_1),\ldots,(b^{(n)},j_n)&\in\{(\omega,k)\ |\ k\geq 2\}\cup\{(\Har,k)\ |\ k\geq 3\},
\end{align}
as was done in the proof of \eqref{eq:0=(omega-3ExB)lE+2lom+2lu=}--\eqref{eq:(Har-1ExB)lE+2lom+2lu=} in Lemma \ref{lemma:r=1-s=3}.
Then, taking the action of each expansion on $\lu$, we have
\begin{align}
\label{eq:w18w48-1}
0&=((\lE+1-\mn)^2-2\mn\omega_{1})
\nonumber\\
&\quad{}\times((16 \mn+3) \lE^2+(-24 \mn^2+58 \mn) \lE
\nonumber\\
&\qquad\quad{}+8 \mn^3+(-8  \omega_{1}-37) \mn^2+(22  \omega_{1}+42) \mn-12  \omega_{1})\ExB_{\lE}\lu
\nonumber\\&\quad{}+
2 \mn (\mn-2) (2 \mn-9) (2 \mn-1)\ExB_{\lE}\Har_{3}\lu,\\
\label{eq:w18w48-2}
0&=
(((\lE+1-\mn)^2-2\mn\omega_{1}))\nonumber\\
&\quad{}\times\big((8 \mn-9) \lE^3+(88 \mn^2+26 \mn-45) \lE^2
\nonumber\\
&\qquad\quad{}+(-146 \mn^3+(16 \omega_{1}+375) \mn^2+(-18 \omega_{1}-43) \mn-54) \lE
\nonumber\\
&\qquad\quad{}+50 \mn^4
+(-60 \omega_{1}-240) \mn^3+(184 \omega_{1}+306) \mn^2
\nonumber\\
&\qquad\qquad{}+(-140 \omega_{1}-40) \mn+24 \omega_{1}-24\big)\ExB_{\lE}\lu
\nonumber\\&\quad{}+
2 (\mn-2) (2 \mn-1) ((6 \mn^2-5 \mn+6) \lE+10 \mn^3-54 \mn^2+10 \mn+6)\ExB_{\lE}\Har_{3}\lu,\\
\label{eq:w18w48-3}
0&=
(((\lE+1-\mn)^2-2\mn\omega_{1}))\nonumber\\
&\quad{}\times\big((16 \mn^2+61 \mn-102) \lE^3+(216 \mn^3+326 \mn^2-67 \mn-510) \lE^2
\nonumber\\
&\qquad\quad{}+(-352 \mn^4+(40 \omega_{1}+349) \mn^3+(-6 \omega_{1}+1772) \mn^2
\nonumber\\
&\qquad\qquad{}+(-88 \omega_{1}-1218) \mn+48 \omega_{1}-540) \lE+120 \mn^5
\nonumber\\
&\qquad\qquad{}+(-144 \omega_{1}-376) \mn^4+
(200 \omega_{1}-405) \mn^3
\nonumber\\
&\qquad\qquad{}+(646 \omega_{1}+1914) \mn^2+(-1180 \omega_{1}-1000) \mn+480 \omega_{1}-240\big)\ExB_{\lE}\lu
\nonumber\\&\quad{}+
2 (\mn-2) (2 \mn-1) \big((14 \mn^3+21 \mn^2-74 \mn+60) \lE+24 \mn^4
\nonumber\\
&\qquad\qquad{}-90 \mn^3-221 \mn^2+220 \mn+60\big)\ExB_{\lE}\Har_{3}\lu,
\end{align}
\begin{align}
\label{eq:w18w48-4}
0&=
5 ((\lE+1-\mn)^2-2\mn\omega_{1}) \nonumber\\
&\quad{}\times\big((12 \mn^2-27 \mn+6) \lE^4\nonumber\\
&\qquad\quad{}+(24 \mn^3+174 \mn^2-501 \mn+114) \lE^3\nonumber\\
&\qquad\quad{}+(1056 \mn^5-1443 \mn^4+(24 \omega_{1}+2373) \mn^3+(-54 \omega_{1}-1317) \mn^2
\nonumber\\
&\qquad\qquad{}+(12 \omega_{1}-1548) \mn+411) \lE^2\nonumber\\
&\qquad\quad{}+(-1584 \mn^6+(352 \omega_{1}+5873) \mn^5+(-814 \omega_{1}-8768) \mn^4\nonumber\\
&\qquad\qquad{}+(1220 \omega_{1}+12088) \mn^3+(-1568 \omega_{1}-7732) \mn^2
\nonumber\\
&\qquad\qquad{}+(756 \omega_{1}-561) \mn-90 \omega_{1}+414) \lE\nonumber\\
&\qquad\quad{}+528 \mn^7+(-704 \omega_{1}-3058) \mn^6+(2820 \omega_{1}+6912) \mn^5\nonumber\\
&\qquad\qquad{}+(48 \omega_{1}^2-5042 \omega_{1}-10000) \mn^4+(-108 \omega_{1}^2+7242 \omega_{1}+11310) \mn^3\nonumber\\
&\qquad\qquad{}+(24 \omega_{1}^2-7052 \omega_{1}-5950) \mn^2+(3060 \omega_{1}+42) \mn-360 \omega_{1}+180\big)\ExB_{\lE}\lu
\nonumber\\&\quad{}+
(\mn-2) (2 \mn-1) \big((1056 \mn^5-582 \mn^4+1428 \mn^3-1932 \mn^2+1992 \mn-450) \lE \nonumber\\
&\qquad\quad{}+792 \mn^6+(176 \omega_{1}-6224) \mn^5+(-392 \omega_{1}+8666) \mn^4
 \nonumber\\
&\qquad\quad{}+(628 \omega_{1}-13729) \mn^3+(-352 \omega_{1}+11014) \mn^2+(192 \omega_{1}+120) \mn-450\big)\ExB_{\lE}\Har_{3}\lu.
\end{align}
We also expand each of $Q^{(4)}_{\lE+4},Q^{(5,1)}_{\lE+5}, Q^{(5,2)}_{\lE+5}$ and $Q^{(6)}_{\lE+6}$
so that the resulting expression is a linear combination of elements of the form 
\begin{align}
	a^{(1)}_{i_1}\cdots a^{(l)}_{i_l}\ExB_{m}b^{(1)}_{j_1}\cdots b^{(n)}_{j_n}
\end{align}
where $l,n\in\Z_{\geq 0}$, $m\in\Z$, and 
\begin{align}
	(a^{(1)},i_1),\ldots,(a^{(l)},i_l)&\in\{(\omega,k)\ |\ k\leq 0\}\cup\{(\Har,k)\ |\ k\leq 2\},\nonumber\\
	(b^{(1)},j_1),\ldots,(b^{(n)},j_n)&\in\{(\omega,k)\ |\ k\geq 1\}\cup\{(\Har,k)\ |\ k\geq 3\}.
\end{align}
Then, taking the action of each expansion on $\lu$, we have
\begin{align}
\label{eq:ExBlE((1+lE)2-2mnomega1)-0}
0&=
\ExB_{\lE}((1+\lE)^2-2 \mn \omega_{1})
\nonumber\\&\quad{}\times{} \big((16 \mn+3)\lE^{2}
+
(-16 \mn^2+36 \mn+12)\lE
\nonumber\\&\qquad\quad{}+
4 \mn^3+(-8 \omega_1-18) \mn^2+(22 \omega_1+14) \mn-12 \omega_1+12\big)\lu
\nonumber\\&\quad{}+\ExB_{\lE}\Har_{3}\big(
8 \mn^4-56 \mn^3+98 \mn^2-36 \mn\big)\lu,
\end{align}
\begin{align}
\label{eq:ExBlE((1+lE)2-2mnomega1)-1}
0&=
\ExB_{\lE}((1+\lE)^2-2 \mn \omega_{1})
\nonumber\\&\quad{}\times{} \big((8 \mn-9)\lE^{3}
\nonumber\\&\qquad\quad{}+
(72 \mn^2+44 \mn-45)\lE^{2}
\nonumber\\&\qquad\quad{}+
(-78 \mn^3+(16 \omega_1+166) \mn^2+(-18 \omega_1+115) \mn-78)\lE
\nonumber\\&\qquad\quad{}+
20 \mn^4+(-60 \omega_1-88) \mn^3+(184 \omega_1+52) \mn^2+(-140 \omega_1+112) \mn+24 \omega_1-48\big)\lu
\nonumber\\&\quad{}+\ExB_{\lE}\Har_{3}\big(
(24 \mn^4-80 \mn^3+98 \mn^2-80 \mn+24)\lE
\nonumber\\&\qquad\quad{}+
40 \mn^5-316 \mn^4+620 \mn^3-292 \mn^2-20 \mn+24\big)\lu,\\
\label{eq:ExBlE((1+lE)2-2mnomega1)-2}
0&=
\ExB_{\lE}((1+\lE)^2-2 \mn \omega_{1})
\nonumber\\&\quad{}\times{} \big((16 \mn^2+61 \mn-102)\lE^{3}
\nonumber\\&\qquad\quad{}+
(176 \mn^3+332 \mn^2+21 \mn-558)\lE^{2}
\nonumber\\&\qquad\quad{}+
(-188 \mn^4+(40 \omega_1+106) \mn^3+(-6 \omega_1+1088) \mn^2\nonumber\\
&\qquad\qquad{}+(-88 \omega_1+74) \mn+48 \omega_1-1068)\lE
\nonumber\\&\qquad\quad{}+
48 \mn^5+(-144 \omega_1-132) \mn^4+(200 \omega_1-282) \mn^3\nonumber\\
&\qquad\qquad{}+(646 \omega_1+678) \mn^2+(-1180 \omega_1+420) \mn+480 \omega_1-720\big)\lu
\nonumber\\&\quad{}+\ExB_{\lE}\Har_{3}\big(
(56 \mn^5-56 \mn^4-450 \mn^3+1064 \mn^2-896 \mn+240)\lE
\nonumber\\&\qquad\quad{}+
96 \mn^6-600 \mn^5+112 \mn^4+2730 \mn^3-2844 \mn^2+280 \mn+240\big)\lu,
	\end{align}
\begin{align}
\label{eq:ExBlE((1+lE)2-2mnomega1)-3}
0&=
\ExB_{\lE}((1+\lE)^2-2 \mn \omega_{1})
\nonumber\\&\quad{}\times{} \big((60 \mn^2-135 \mn+30)\lE^{4}
\nonumber\\&\qquad\quad{}+
(1140 \mn^2-2565 \mn+570)\lE^{3}
\nonumber\\&\qquad\quad{}+
(3520 \mn^5-2845 \mn^4+(120 \omega_1+4970) \mn^3+(-270 \omega_1+1675) \mn^2\nonumber\\
&\qquad\qquad{}+(60 \omega_1-11580) \mn+2505)\lE^{2}
\nonumber\\&\qquad\quad{}+
(-3520 \mn^6+(1760 \omega_1+11230) \mn^5+(-4550 \omega_1-10490) \mn^4\nonumber\\
&\qquad\qquad{}+(7180 \omega_1+13010) \mn^3+(-8080 \omega_1+6570) \mn^2\nonumber\\
&\qquad\qquad{}+(3780 \omega_1-22110) \mn-450 \omega_1+4320)\lE
\nonumber\\&\qquad\quad{}+
(880 \mn^7+(-3520 \omega_1-4660) \mn^6+(14340 \omega_1+7480) \mn^5\nonumber\\
&\qquad\qquad{}+(240 \omega_1^2-26230 \omega_1-5875) \mn^4+(-540 \omega_1^2+37410 \omega_1+2050) \mn^3\nonumber\\
&\qquad\qquad{}+(120 \omega_1^2-35500 \omega_1+13280) \mn^2+(15300 \omega_1-15990) \mn-1800 \omega_1+2700)
\big)\lu
\nonumber\\&\quad{}+\ExB_{\lE}\Har_{3}\big(
(1760 \mn^7-4780 \mn^6+4310 \mn^5-7540 \mn^4+13100 \mn^3-13060 \mn^2+5850 \mn-900)\lE
\nonumber\\&\qquad\quad{}+
(1760 \mn^8+(352 \omega_1-17592) \mn^7+(-1664 \omega_1+53484) \mn^6\nonumber\\
&\qquad\qquad{}+(3568 \omega_1-89118) \mn^5+(-4628 \omega_1+114333) \mn^4+(3400 \omega_1-86520) \mn^3\nonumber\\
&\qquad\qquad{}+(-1664 \omega_1+22384) \mn^2+(384 \omega_1+2106) \mn-900)
\big)\lu.
\end{align}
Note that $\omega_{1}$'s are on the left side of $\ExB_{\lE}$ in \eqref{eq:w18w48-1}--\eqref{eq:w18w48-4},
but are on the right side of $\ExB_{\lE}$ in \eqref{eq:ExBlE((1+lE)2-2mnomega1)-0}--\eqref{eq:ExBlE((1+lE)2-2mnomega1)-3}.

\begin{enumerate}
\item
Assume $\Har_{3}\lu=0$
, $\lvE=(\omega_1-(\lE+1-\mn)^2/(2\mn))E_{\lE}\lu\neq 0$, and $\lE\neq \mn-2$.
By \eqref{eq:w18w48-1},
\begin{align}
\label{eq:dfrac3lE2+42mn +58lEmn}
0&=(2(\mn-2)(4\mn-3)\omega_1\nonumber\\
&\qquad{}-(3 \lE^2+42 \mn +58 \lE \mn +16 \laE^2 \mn -37 \mn ^2-24 \lE \mn ^2+8 \mn ^3))
\ExB_{\lE}\lv.
\end{align}
Then, by \eqref{eq:w18w48-2} and \eqref{eq:w18w48-3},
\begin{align}
0&=(10 \mn^2-4 \mn+3)(2 t-\mn+2) ((8 \mn-9) t-4 \mn^2+16 \mn-12),\nonumber\\
0&=(12 \mn^3+16 \mn^2-35 \mn+15) (2 t-\mn+2) ((8 \mn-9) t-4 \mn^2+16 \mn-12)
\end{align}
and hence 
\begin{align}
\label{eq:0=mn92mn2mnlE}
0&=(2 t-\mn+2) ((8 \mn-9) t-4 \mn^2+16 \mn-12).
\end{align}
By \eqref{eq:w18w48-3} and \eqref{eq:dfrac3lE2+42mn +58lEmn},
\begin{align}
\label{eq:(4032 mn5-2880 mn4+3360 mn3-2091 mn2+774 mn-108}
0&=
(4032 \mn^5-2880 \mn^4+3360 \mn^3-2091 \mn^2+774 \mn-108) t^3\nonumber\\
&\quad{}+(11264 \mn^7-38336 \mn^6+71692 \mn^5-87577 \mn^4\nonumber\\
&\qquad{}+82959 \mn^3-43593 \mn^2+12078 \mn-1431) t^2\nonumber\\
&\quad{}+(-11264 \mn^8+62768 \mn^7-146980 \mn^6+228928 \mn^5\nonumber\\
&\qquad{}-266274 \mn^4+218205 \mn^3-103356 \mn^2+26244 \mn-2916) t\nonumber\\
&\quad{}+2816 \mn^9-22304 \mn^8+72152 \mn^7-138308 \mn^6\nonumber\\
&\qquad{}+193478 \mn^5-203897 \mn^4+147876 \mn^3-64152 \mn^2+15282 \mn-1620.
\end{align}
If $2 t-\mn+2=0$, then $\lE\neq 0$ and by \eqref{eq:(4032 mn5-2880 mn4+3360 mn3-2091 mn2+774 mn-108}
\begin{align}
0&=t^2(4t+3)(352t^4+1356t^3+2080t^2+1452t+385),
\end{align}
a contradiction.
If $(8 \mn-9) t-4 \mn^2+16 \mn-12=0$, then as polynomials in $\mn$,
the right-hand side of \eqref{eq:(4032 mn5-2880 mn4+3360 mn3-2091 mn2+774 mn-108}
divided by $(8 \mn-9) t-4 \mn^2+16 \mn-12$ leaves a remainder of 
\begin{align}
\label{eq:(10560 mn-11880) t6+(76320 mn-169245/2) t5}
0&=(10560 \mn-11880) t^6+(76320 \mn-169245/2) t^5\nonumber\\
&\quad{}+(471705/2 \mn-2059785/8) t^4+(3187935/8 \mn-6397605/16) t^3\nonumber\\
&\quad{}+(3012585/8 \mn-5600025/16) t^2+(202005 \mn-356265/2) t\nonumber\\
&\quad{}+45900 \mn-39285.
\end{align}
Moreover, $(8 \mn-9) t-4 \mn^2+16 \mn-12$ divided by  the right-hand side of 
\eqref{eq:(10560 mn-11880) t6+(76320 mn-169245/2) t5} leaves a remainder of 
\begin{align}
0&=
\frac{-9 (t+1) (4 t+3)^2 (7 t+4) (7 t^2+10 t+7)}
{(5632 t^6+40704 t^5+125788 t^4+212529 t^3+200839 t^2+107736 t+24480)^2} \nonumber\\
&\quad{}\times (94556 t^4+495381 t^3+1010052 t^2+931824 t+326592)
\end{align}
and hence $\lE=-1$, which implies $\mn=1/2$ by \eqref{eq:(10560 mn-11880) t6+(76320 mn-169245/2) t5}, a contradiction.
 Thus, 
$\lE=\mn-2$
by \eqref{eq:0=mn92mn2mnlE} and \eqref{eq:omega1lv=lv} holds by \eqref{eq:dfrac3lE2+42mn +58lEmn}.
\item Assume $\omega_{1}\lu=\lu$ and $\Har_{3}\lu=\lu$.
By \eqref{eq:ExBlE((1+lE)2-2mnomega1)-0} 
and \eqref{eq:ExBlE((1+lE)2-2mnomega1)-1},
$tg_1(\mn)=tg_2(\mn)=0$ where
\begin{align}
\label{eq:g1(mn)=3(lE+1)2(lE+4)+}
g_1(\mn)&=3 (\lE+1)^2 (\lE+4)+
2 (4 \lE+7) (2 \lE^2+5 \lE+6)\mn\nonumber\\
&\quad{}-2 (8 \lE^2+45 \lE+70)\mn^{2}+4 (\lE+10)\mn^{3},\nonumber\\
g_2(\mn)&=-3 (\lE+1) (\lE+2) (3 \lE^2+12 \lE+17)\nonumber\\
&\quad{}+(8 \lE^4+60 \lE^3+211 \lE^2+300 \lE+117)\mn\nonumber\\
&\quad{}+2 (36 \lE^3+155 \lE^2+292 \lE+279)\mn^{2}\nonumber\\
&\quad{}-2 (39 \lE^2+224 \lE+409)\mn^{3}\nonumber\\
&\quad{}+20 (\lE+11)\mn^{4}.
\end{align}
We note that the degrees of $g_1(\mn)$ and $g_2(\mn)$
in $\mn$ are at most $3$ and $4$ respectively.
Assume $\lE\neq 0$. By \eqref{eq:g1(mn)=3(lE+1)2(lE+4)+}, 
\begin{align}
0&=2 (10 +\lE)^{2}g_1(\mn)-(-240 + 176 \lE +  51 \lE^2 +\lE^3+ 10(11 + \lE)(10+\lE) \mn)
g_2(\mn)
\nonumber\\
&= -3 (\lE+1) (7 \lE^5+212 \lE^4+1837 \lE^3+6152 \lE^2+9064 \lE+5840)\nonumber\\
&\quad{}-6 (79 \lE^5+465 \lE^4+90 \lE^3-3352 \lE^2-7666 \lE-5060)\mn\nonumber\\
&\quad{}+6 (61 \lE^4+409 \lE^3+866 \lE^2-880 \lE-2400)\mn^{2},
\end{align}
which is a polynomial of degree at most $2$ in $\mn$. 
Repeating this procedure to decrease the degrees of polynomials in $\mn$,
we finally obtain
\begin{align}
	\label{eq:g1(lE)=-3(-1 + lE)}
0&=(-1 + \lE) (1 + \lE) (2 + \lE)^2 (3 + \lE) 
	(-20 + \lE + 33 \lE^2) \nonumber\\
	&\quad{}\times (433 + 235 \lE + 	67 \lE^2 + 33 \lE^3)\nonumber\\
	&\quad{}\times  (-2400 - 880 \lE + 	866 \lE^2 + 409 \lE^3 + 61 \lE^4)^2 \nonumber\\
&\quad{}\times 	(5020 + 14072 \lE + 18476 \lE^2 + 
	13940 \lE^3 + 6272 \lE^4 + 1580 \lE^5 + 	175 \lE^6).
	\end{align}
Since $\lE$ is an integer, it follows from
 \eqref{eq:ExBlE((1+lE)2-2mnomega1)-0}, 
\eqref{eq:ExBlE((1+lE)2-2mnomega1)-1}, 	and \eqref{eq:g1(lE)=-3(-1 + lE)}
that $(\lE,\mn)=(-1,0),(1,2),(-3,2)$, or $(-2,1/2)$, a contradiction.
 Thus, $\lE=0$.
If $\omega_1\lu=(1/16)\lu$ (resp. $(9/16)\lu$) and $H_{3}\lu=(-1/128)\lu$ (resp. $(15/128)\lu$), then 
the same argument as above shows the results.
\end{enumerate}
\end{proof}

\begin{lemma}
\label{lemma:structure-Vlattice-M1-norm2}
Assume $\mn=2$.
Let $\lu$ be a non-zero element of $\Omega_{M(1)^{+}}(\mW)$ with $I(E,x)\lu\neq 0$.
We write 
\begin{align}
\lE&=\epsilon(\ExB,\lu)
\end{align}
for simplicity.
Then, 
\begin{align}
	\label{eq:(H0ElE1lu=dfrac194+7lE}
	(H_{0}E)_{\lE+3}\lu&=
	\dfrac{-1}{9 (4+7 \lE)}E_{\lE}
	(36+72 \Har_{3}+142 \lE+123 \lE^2+18 \lE^3\nonumber\\
	&\qquad{}+\lE^4+156 \omega_{1}-40 \lE \omega_1+8 \lE^2\omega_1-48 \omega_1^2)\lu.
\end{align}
We set
\begin{align}
\lv&=(\omega_1-\frac{(\lE+1)^2}{4})\lu.
\end{align}
We have
\begin{align}
E_{\lE}\lv=(\omega_1-\frac{(\lE-1)^2}{4})E_{\lE}\lu.
\end{align}
\begin{enumerate}
\item
Assume $\Har_{3}\lu=0$.
If $\ExB_{\lE}\lv\neq 0$, then $\lE=0$ and
\begin{align}
\omega_1\ExB_{0}\lv&=\ExB_{0}\lv.
\end{align}
If $\lu$ is an eigenvector of $\omega_1$ and $\lv\neq 0$, then 
\begin{align}
\label{eq:omega1omega11lu=0}
\omega_1\lu&=\lu.
\end{align}
\item
Let $(\zeta,\xi)\in\{(1,1),(1/16,-1/128),(9/16,15/128)\}$.
If $\omega_1\lu=\zeta\lu$ and $\Har_{3}\lu=\xi\lu$, then
$\lE=0$. 
\end{enumerate}
\end{lemma}
\begin{proof}
	Taking the $(\lE+3)$-th action of \eqref{eq:6omega2E4omega0omega1E-1},
the $(\lE+4)$-th action of \eqref{eq:6omega2E4omega0omega1E-2}, 
the $(\lE+5)$-th action of \eqref{eq:6omega2E4omega0omega1E-3}, and 
the $(\lE+7)$-th action of \eqref{eq:6omega2E4omega0omega1E-4} on $\lu$,
we have 
\begin{align}
0&=\lE((1-\lE)^2-4\omega_1)E_{\lE}\lu
\label{eq:lE1lE24omega1ElElu-0}\\
&
=E_{\lE}\lE((1+\lE)^2-4\omega_1)\lu,
\label{eq:lE1lE24omega1ElElu-0-1}\\
0&=9(4 + 7 \lE)(H_{0}E)_{\lE+3}\lu+72 E_{\lE}\Har_{3}\lu\nonumber\\
&\quad{}+(36 + 298 \lE + 35 \lE^2 + 26 \lE^3 + \lE^4 \nonumber\\
&\qquad{}+ 156 \omega_{1} - 136 \lE \omega_{1} + 
 8 \lE^2 \omega_{1} - 48 \omega_{1}^2)E_{\lE}\lu
\label{eq:lE1lE24omega1ElElu-1}\\
&=\ExB_{\lE}\Big((t+1) ( \lE ^3+17  \lE ^2+106  \lE +36) +4 (2  \lE ^2-10  \lE +39)\omega_{1 } -48 \omega_{
1 }^2\Big)\lu\nonumber\\&\quad{}
+9 (7  \lE +4)(H_{0}E)_{ \lE +3 }\lu +72E_{ \lE  } H_{3 }\lu,
\label{eq:lE1lE24omega1ElElu-4}\\
0&=9 (-520 - 959 \lE + 33 \lE^2 - 100 \omega_{1})(H_{0}E)_{\lE+3}\lu+72  (-155 + 44 \lE) E_{\lE}\Har_{3}\lu\nonumber\\
&\quad{}-8 (585 + 4492 \lE + 1576 \lE^2 - 210 \lE^3 + 62 \lE^4 \nonumber\\
&\qquad{}+ 2 \lE^5 + 2685 \omega_{1} + 32 \lE \omega_{1} - 192 \lE^2 \omega_{1} + 16 \lE^3 \omega_{1} \nonumber\\
&\qquad{}- 480 \omega_{1}^2 -   96 \lE \omega_{1}^2)E_{\lE}\lu,
\label{eq:lE1lE24omega1ElElu-3}\\
&=E_{ \lE  }\Big(
-8 ( \lE +1) (2  \lE ^4+44  \lE ^3-158  \lE ^2+1222  \lE +585)\nonumber\\
&\qquad{}-8 (16  \lE ^3+992  \lE +2685) \omega_{1 }
+768 ( \lE +5) \omega_{1 }^2\Big)\lu\nonumber\\
&\quad{}+72 (44  \lE -155)E_{ \lE  } H_{3 }\lu \nonumber\\
&\quad{}+(H_{0}E)_{ \lE +3 }\big(9 (33  \lE ^2-859  \lE -520) -900\omega_{1 }\big)\lu,
\label{eq:lE1lE24omega1ElElu-5}
\end{align}
\begin{align}
0&=
-8 (52787700 t^7+1097587588 t^6+5494080415 t^5-89625113568 t^4+68909700044 t^3\nonumber\\
&\qquad{}+2468574039524 t^2+3786872840265 t+493804109430)\ExB_{t } \lu
\nonumber\\&\quad{}
+16 (52787700 t^5+336785348 t^4+16075086171 t^3-110729180408 t^2\nonumber\\
&\qquad{}-710794593411 t-1166720253525)\omega_{1 } \ExB_{t } \lu
\nonumber\\&\quad{}
+3 (5886227459 t^4+64230119866 t^3-465363710675 t^2-2778231175402 t\nonumber\\
&\qquad{}-1316641482180)(\Har_{0}\ExB)_{t+3 } \lu
\nonumber\\&\quad{}
+72 (743028209 t^3+17731219498 t^2+23020475889 t-140972980110)\ExB_{t } \Har_{3 } \lu
\nonumber\\&\quad{}
-60 (146187286 t^2+9549468148 t+19688188167)\omega_{1 } (\Har_{0}\ExB)_{t+3 } \lu
\nonumber\\&\quad{}
-28394226000\omega_{1 } \omega_{1 } (\Har_{0}\ExB)_{t+3 } \lu
\nonumber\\&\quad{}
+1536 (93467197 t^2+449390927 t+1282501650)\omega_{1 } \omega_{1 } \ExB_{t } \lu
\nonumber\\&\quad{}
+14515200 (931 t+661)\omega_{1 } \omega_{1 } \omega_{1 } \ExB_{t } \lu
\nonumber\\&\quad{}
+67132800 (931 t+661)\omega_{1 } \ExB_{t } \Har_{3 } \lu
\nonumber\\&\quad{}
+35590023000(\Har_{0}\ExB)_{t+3 } \Har_{3 } \lu \label{eq:lE1lE24omega1ElElu-6}\\
&=
-8 (t+1) (52787700 t^6+1150375288 t^5+5017275823 t^4-78748712473 t^3\nonumber\\
&\qquad{}-158883687883 t^2+959628223785 t+493804109430)\ExB_{t }\lu \nonumber\\&\quad{}
+16 (52787700 t^5+336785348 t^4+663193947 t^3-195213260792 t^2\nonumber\\
&\qquad{}-957034910211 t-1166720253525)\ExB_{t } \omega_{1 }\lu \nonumber\\&\quad{}
+3 (5886227459 t^4+67153865586 t^3-283839089715 t^2-2384467412062 t\nonumber\\
&\qquad{}-1316641482180)(\Har_{0}\ExB)_{t+3 }\lu \nonumber\\&\quad{}
+72 (743028209 t^3+16863155098 t^2+22404159489 t-140972980110)\ExB_{t } \Har_{3 }\lu \nonumber\\&\quad{}
-60 (146187286 t^2+8602993948 t+19688188167)(\Har_{0}\ExB)_{t+3 } \omega_{1 }\lu \nonumber\\&\quad{}
-28394226000(\Har_{0}\ExB)_{t+3 } \omega_{1 }^2\lu\nonumber\\&\quad{}
+1536 (67073347 t^2+430651577 t+1282501650)\ExB_{t } \omega_{1 }^2\lu  \nonumber\\&\quad{}
+14515200 (931 t+661)\ExB_{t } \omega_{1 }^3\lu \nonumber\\&\quad{}
+67132800 (931 t+661)\ExB_{t } \Har_{3 } \omega_{1 }\lu \nonumber\\&\quad{}
+35590023000(\Har_{0}\ExB)_{t+3 } \Har_{3 }\lu.\label{eq:lE1lE24omega1ElElu-7}
\end{align}
Note that $\omega_{1}$'s are on the left side of $\ExB_{\lE}$ in 
\eqref{eq:lE1lE24omega1ElElu-0},
\eqref{eq:lE1lE24omega1ElElu-1}, \eqref{eq:lE1lE24omega1ElElu-3}, and \eqref{eq:lE1lE24omega1ElElu-6},
but are on the right side of $\ExB_{\lE}$ in 
\eqref{eq:lE1lE24omega1ElElu-0-1},
\eqref{eq:lE1lE24omega1ElElu-4}, \eqref{eq:lE1lE24omega1ElElu-5}, and \eqref{eq:lE1lE24omega1ElElu-7}.
By Lemma \ref{lemma:bound-H0E} and \eqref{eq:lE1lE24omega1ElElu-1}, we have \eqref{eq:(H0ElE1lu=dfrac194+7lE}.
Deleting the terms including $(H_{0}E)_{\lE+3}\lu$
from the simultaneous equations \eqref{eq:lE1lE24omega1ElElu-1}, \eqref{eq:lE1lE24omega1ElElu-3}, and \eqref{eq:lE1lE24omega1ElElu-6},
we have 
\begin{align}
0&=360(-4 + 2 \lE + 11 \lE^2 + 4 \omega_{1})E_{\lE}\Har_{3}\lu\nonumber\\
&\quad{}+((1 - \lE)^2 - 4\omega_{1})
\big(
-\lE (-2596 - 5354 \lE + 745 \lE^2 + 29 \lE^3)\nonumber\\
&\qquad{}-4 (60 - 341 \lE + 82 \lE^2)\omega_{1}-240 \omega_{1}^2
\big)E_{\lE}\lu,\nonumber\\
0&=
-360 (-4 + 2 \lE + 11 \lE^2 + 4 \omega_{1})\ExB_{\lE}\Har_{3}\lu\nonumber\\
&\quad{}+((1 - \lE)^2 - 4\omega_{1}) \big(29 \lE^4+745 \lE^3+(328 \omega_{1}-5354) \lE^2\nonumber\\
&\qquad{}+(-1364 \omega_{1}-2596) \lE-240 \omega_{1}^2+240 \omega_{1}\big)\ExB_{\lE} \lu.
\label{eq:lE47lE1lE24omega1-2}
\end{align}
By \eqref{eq:lE1lE24omega1ElElu-0}, $\lE=0$ or $((1 - \lE)^2-4\omega_{1})E_{\lE}\lu=0$.
If $\lE=0$ and  $\Har_{3}\lu=0$, then by \eqref{eq:lE47lE1lE24omega1-2},
\begin{align}
0&=(\omega_1-1)(\omega_1-\dfrac{1}{4})E_{\lE}\lu,
\end{align}
which finishes (1).

By using \eqref{eq:lE1lE24omega1ElElu-0-1}, \eqref{eq:lE1lE24omega1ElElu-4}, \eqref{eq:lE1lE24omega1ElElu-5}, 
 and \eqref{eq:lE1lE24omega1ElElu-7},
the same argument as above shows (2).
\end{proof}

\begin{lemma}
	\label{lemma:structure-Vlattice-M1-norm1-2}
	Assume $\mn=1/2$.
	Let $U$ be an $A(M(1)^{+})$-submodule of $\Omega_{M(1)^{+}}(\mW)$,
	$\lu$ a simultaneous eigenvector of $\{\omega_1,\Har_{3}\}$ in $U$ with $I(E,x)\lu\neq 0$.
	We write 
	\begin{align}
		\lE&=\epsilon(\ExB,\lu)
	\end{align}
	for simplicity. Then,
	\begin{align}
		\label{eq:omega1lu=(1+lE)2lu}
		\omega_1\lu&=(1+\lE)^2\lu.
	\end{align}
	If	$\lE\neq -1$, then
	\begin{align}
		\label{eq:(Har0E)lE+3lu=ExB-0}
		(\Har_{1}E)_{\lE+2}\lu&=\ExB_{\lE}\big(\frac{3}{(1 + \lE)
			(3 + 2\lE)}\Har_{3} + (1 + \lE)\big)\lu,
	\end{align} 
	and if $\lE=-1$, then
	\begin{align}
		\label{eq:(Har0E)lE+3lu=ExB-1}
		\Har_{3}\lu&=0.
	\end{align} 
\end{lemma}
\begin{proof}
	Taking the $(\lE+2)$-th action of \eqref{eq:norm1-2-0} on $\lu$, we have
	\begin{align}
		\omega_{1}\ExB_{\lE}\lu&=(\lE+\frac{1}{2})^2\ExB_{\lE}\lu
	\end{align}
	and hence \eqref{eq:omega1lu=(1+lE)2lu} holds.
	Taking the $(\lE+3)$-th action of \eqref{eq:norm1-2-1}
	and the $(\lE+4)$-the action of \eqref{eq:norm1-2-2} on $\lu$, we have
	\begin{align}
		0&=
		8(\lE+1)E_{\lE}\omega_{1 }\lu -(\lE+1)(11\lE+15)(\Har_{1}E)_{\lE+2}\lu \nonumber\\&\quad{}
		+4(\lE+1)^2E_{\lE}\lu +12E_{\lE}\Har_{3 } \lu 
		+3(\Har_{1}E)_{\lE+2}\omega_{1 }\lu,
	\end{align}
	and hence \eqref{eq:(Har0E)lE+3lu=ExB-0}
	and \eqref{eq:(Har0E)lE+3lu=ExB-1} by \eqref{eq:omega1lu=(1+lE)2lu}.
\end{proof}

Combining Lemmas \ref{lemma:r=1-s=3}, \ref{lemma:structure-Vlattice-M1}, \ref{lemma:structure-Vlattice-M1-norm2}, \ref{lemma:structure-Vlattice-M1-norm1-2},
and \cite[Theorem 6.2]{DLM1998t},
we have the following result:
\begin{proposition}
\label{proposition:Zhu-Omega}
Let $\lattice$ be a non-degenerate even lattice of rank $1$
and $\module$ a non-zero weak $V_{\lattice}^{+}$-module.
Then, there exists an irreducible $A(M(1)^{+})$-submodule 
of $\Omega_{M(1)^{+}}(M)$. In particular, there exists a non-zero $M(1)^{+}$-submodule of $\module$. 
\end{proposition}
\appendix
\renewcommand{\thesection}{Appendix A\arabic{section}}
\renewcommand{\thesubsection}{A\arabic{section}-\arabic{subsection}}
\section{}
\label{section:appendix}
\renewcommand{\thesection}{A\arabic{section}}
In this appendix, for some $a,b\in V_{\lattice}$, we put the computations of 
$a_{k}b$ for $k\in\Z_{\geq 0}$.
For $k\in\Z_{\geq 0}$ not listed below, $a_{k}b=0$.
Using these results, we can compute the commutation relation $[a_i,b_j]=\sum_{k=0}^{\infty}\binom{i}{k}(a_kb)_{i+j-k}$.

\subsection{Computations in $M(1)$}
\label{section:normal-total}
\label{section:normal-first}
\begin{align}
	\omega_{0}\omega&=
	\omega_{0 } \omega _{-1 } \vac
	,
&
	\omega_{1}\omega&=
	2\omega _{-1 } \vac
	,
&
	\omega_{2}\omega&=0,&
	\omega_{3}\omega&=
	\frac{1}{2}\vac
	,
\end{align}

\begin{align}
	\omega_{0}\Har&=
	\omega_{0 } \Har _{-1 } \vac
	,
&	\omega_{1}\Har&=
	4\Har _{-1 } \vac
	,
&	\omega_{2}\Har&=
	\frac{-1}{3}\omega_{0 } \omega _{-1 } \vac
	,
	\nonumber\\
	\omega_{3}\Har&=
	2\omega _{-1 } \vac
	,
&	\omega_{4}\Har&=0,&
	\omega_{5}\Har&=
	\frac{-1}{3}\vac
	,
\end{align}
\begin{align}
	\label{eq:Har0Har=2omega0}
	\Har_{0}\Har&=2\omega_{0 } \omega _{-2 } \omega _{-2 } \vac
	+\frac{24}{5}\omega_{0 } \omega _{-1 } \Har _{-1 } \vac
	\nonumber\\&\quad{}
	+\frac{-2}{5}\omega_{0 }^3 \omega _{-1 } \omega _{-1 } \vac
	+\frac{-3}{10}\omega_{0 }^3 \Har _{-1 } \vac
	+\frac{1}{20}\omega_{0 }^5\omega _{-1 } \vac,\nonumber\\
	\Har_{1}\Har&=
	4\omega _{-2 } \omega _{-2 } \vac
	+\frac{48}{5}\omega _{-1 } \Har _{-1 } \vac
	\nonumber\\&\quad{}
	+\frac{-4}{5}\omega_{0 }^2 \omega _{-1 } \omega _{-1 } \vac
	+\frac{16}{15}\omega_{0 }^2 \Har _{-1 } \vac
	+\frac{7}{45}\omega_{0 }^4\omega _{-1 } \vac,\nonumber\\
	\Har_{2}\Har&=
	10\omega_{0 } \Har _{-1 } \vac
	+\frac{7}{18}\omega_{0 }^3 \omega _{-1 } \vac,\nonumber\\
	\Har_{3}\Har&=
	20\Har_{-1 } \vac
	+\frac{4}{3}\omega_{0 }^2 \omega_{-1 } \vac,\nonumber\\
	\Har_{4}\Har&=
	\frac{10}{3}\omega_{0 } \omega_{-1 } \vac,\nonumber\\
	\Har_{5}\Har&=\frac{20}{3}\omega_{-1 } \vac,\nonumber\\
	\Har_{6}\Har&=0,\nonumber\\
	\Har_{7}\Har&=\frac{5}{3}\vac.
\end{align}

\subsection{The case that $\langle \alpha,\alpha\rangle\neq 0,1/2,1$, and $2$}
\label{section:Norm-nonzero}
Let $\mn$ be a complex number with $\mn\neq 0,1/2,1,2$ and $\alpha\in\fh$ such that $\langle\alpha,\alpha\rangle=\mn$.

\begin{align}
	\label{eq:Har[1]0ExB=frac2mnmn-2omega-2E}
\Har_{0}\ExB&=
\frac{2  \mn }{\mn -2}\omega_{-2 } E
+\frac{-4  \mn }{(\mn -2)  (2  \mn -1)}\omega_{0 } \omega_{-1 } E
+\frac{2}{(\mn -2)  (2  \mn -1)}\omega_{0 }^{3}E,\nonumber\\
\Har_{1}\ExB&=
\frac{2  \mn }{2  \mn -1}\omega_{-1 } E
+\frac{-1}{2  \mn -1}\omega_{0 } \omega_{0 } E,
\nonumber\\
\Har_{2}\ExB&=
\frac{1}{3}\omega_{0 } E.
\end{align}

\subsection{The case that $\langle \alpha,\alpha\rangle=2$}
\label{section:normal-2}
Let $\alpha\in\fh$ with $\langle\alpha,\alpha\rangle=2$.
\begin{align}
\label{eq:omega[1]0Har[1]0ExB=omega0}
\omega_{0}\Har_{0}\ExB&=
\omega_{0 } (\Har_{0}E),
\quad
\omega_{1}\Har_{0}\ExB=
4(\Har_{0}E),
\quad
\omega_{2}\Har_{0}\ExB=
8\omega_{-1 } E
-2\omega_{0 } \omega_{0 } E,
\nonumber\\
\omega_{3}\Har_{0}\ExB&=
6\omega_{0 } E,
\quad
\omega_{4}\Har_{0}\ExB=
12E,
\end{align}
\begin{align}
\Har_{0}\ExB&=
\Har_{0}E,
\quad
\Har_{1}\ExB=
\frac{4}{3}\omega_{-1 } E
+\frac{-1}{3}\omega_{0 } \omega_{0 } E,
\quad
\Har_{2}\ExB=
\frac{1}{3}\omega_{0 } E,
\end{align}

\begin{align}
\Har_{0}\Har_{0}\ExB&=
\frac{4096}{5145}\omega_{-1 } \omega_{-1 } \omega_{-1 } E
+\frac{210412}{46305}\omega_{-2 } (\Har_{0}E) \nonumber\\&\quad{}
+\frac{18944}{5145}\omega_{-1 } \Har_{-1 } E
+\frac{-124504}{15435}\Har_{-3 } E\nonumber\\&\quad{}
+\frac{5888}{36015}\omega_{0 } \omega_{-1 } (\Har_{0}E) 
+\frac{3032168}{324135}\omega_{0 } \Har_{-2 } E\nonumber\\&\quad{}
+\frac{-42680128}{72930375}\omega_{0 } \omega_{0 } \omega_{-1 } \omega_{-1 } E
+\frac{-42960376}{24310125}\omega_{0 } \omega_{0 } \Har_{-1 } E\nonumber\\&\quad{}
+\frac{-4044646}{24310125}\omega_{0 } \omega_{0 } \omega_{0 } (\Har_{0}E) 
+\frac{32226464}{218791125}\omega_{0 } \omega_{0 } \omega_{0 } \omega_{0 } \omega_{-1 } E\nonumber\\&\quad{}
+\frac{-2775692}{218791125}\omega_{0 } \omega_{0 } \omega_{0 } \omega_{0 } \omega_{0 } \omega_{0 } E,\end{align}

\begin{align}
\Har_{1}\Har_{0}\ExB&=
\frac{36}{7}\omega_{-1 } (\Har_{0}E) 
+\frac{80}{7}\Har_{-2 } E
+\frac{1088}{1575}\omega_{0 } \omega_{-1 } \omega_{-1 } E
+\frac{1496}{525}\omega_{0 } \Har_{-1 } E\nonumber\\&\quad{}
+\frac{-709}{525}\omega_{0 } \omega_{0 } (\Har_{0}E) 
+\frac{-544}{4725}\omega_{0 } \omega_{0 } \omega_{0 } \omega_{-1 } E\nonumber\\&\quad{}
+\frac{-68}{4725}\omega_{0 } \omega_{0 } \omega_{0 } \omega_{0 } \omega_{0 } E,
\nonumber\\
\Har_{2}\Har_{0}\ExB&=
\frac{128}{25}\omega_{-1 } \omega_{-1 } E
+\frac{528}{25}\Har_{-1 } E
+\frac{-11}{75}\omega_{0 } (\Har_{0}E) \nonumber\\&\quad{}
+\frac{-64}{75}\omega_{0 } \omega_{0 } \omega_{-1 } E
+\frac{-8}{75}\omega_{0 } \omega_{0 } \omega_{0 } \omega_{0 } E,
\nonumber\\
\Har_{3}\Har_{0}\ExB&=
27(\Har_{0}E) ,
\nonumber\\
\Har_{4}\Har_{0}\ExB&=
48\omega_{-1 } E-12\omega_{0 } \omega_{0 } E,
\nonumber\\
\Har_{5}\Har_{0}\ExB&=
20\omega_{0 } E.
\end{align}

\subsection{The case that $\langle \alpha,\alpha\rangle=1/2$}
\label{section:norm-1-2-part1}
Let $\alpha\in\fh$ with $\langle\alpha,\alpha\rangle=1/2$.
\begin{align}
\omega_{0}\ExB&=
\omega_{0 } E,\quad
\omega_{1}\ExB=
\frac{1}{4}E,
\end{align}

\begin{align}
\label{eq:norm1-2-omega[1]0Har[1]1ExB}
\omega_{0}\Har_{1}\ExB&=
\omega_{0 } (\Har_{1}\ExB),\quad
\omega_{1}\Har_{1}\ExB=
\frac{9}{4}(\Har_{1}\ExB),\quad
\omega_{2}\Har_{1}\ExB=
2\omega_{0 } E,
\quad
\omega_{3}\Har_{1}\ExB=E,
\end{align}

\begin{align}
\Har_{0}\ExB&=
\frac{-2}{3}\omega_{-2 } E+\frac{4}{3}\omega_{0 } (\Har_{1}\ExB),
\quad
\Har_{1}\ExB=
(\Har_{1}\ExB),
\quad
\Har_{2}\ExB=
\frac{1}{3}\omega_{0 } E,
\end{align}

\begin{align}
\Har_{0}\Har_{1}\ExB&=
\frac{-172}{25}\Har_{-2 } E+\frac{42}{25}\omega_{-2 } (\Har_{1}\ExB)
+\frac{126}{25}\omega_{0 } \Har_{-1 } E+\frac{22}{25}\omega_{0 } \omega_{-1 } (\Har_{1}\ExB)\nonumber\\&\quad{}
+\frac{-4}{75}\omega_{0 } \omega_{0 } \omega_{-2 } E+\frac{-58}{75}\omega_{0 } \omega_{0 } \omega_{0 } (\Har_{1}\ExB),
\nonumber\\
\Har_{1}\Har_{1}\ExB&=
\frac{-3}{2}\Har_{-1 } E+\frac{5}{2}\omega_{-1 } (\Har_{1}\ExB)
+\frac{-11}{3}\omega_{0 } \omega_{-2 } E+\frac{29}{6}\omega_{0 } \omega_{0 } (\Har_{1}\ExB),
\nonumber\\
\Har_{2}\Har_{1}\ExB&=
-4\omega_{-2 } E+\frac{25}{3}\omega_{0 } (\Har_{1}\ExB),
\nonumber\\
\Har_{3}\Har_{1}\ExB&=
8(\Har_{1}\ExB),
\nonumber\\
\Har_{4}\Har_{1}\ExB&=
4\omega_{0 } E,
\nonumber\\
\Har_{5}\Har_{1}\ExB&=
\frac{1}{3}E.
\end{align}

\

\renewcommand{\thesection}{Notation}

\section{}\label{section:notation}
\begin{tabular}{lp{13cm}}
$V$ & a vertex algebra.\\
$U$ & a subspace of a weak $V$-module.\\
$\Omega_{V}(U)$&$
=\{\lu\in U\ \Big|\ 
a_{i}\lu=0\ \mbox{for all homogeneous }a\in V\mbox{and }i>\wt a-1.\}$ \eqref{eq:OmegaV(U)=BigluinU}.\\
$\mn$ & a non-zero complex number.\\
$\hei$ & a finite dimensional vector space equipped with a nondegenerate symmetric bilinear form
$\langle \mbox{ }, \mbox{ }\rangle$.\\
$h$ & an element of $\fh$ with $\langle \wh,\wh\rangle=1$.\\
$h^{[1]},\ldots,h^{[\rankL]}$ & an orthonormal basis of $\fh$.\\
$\alpha$ & an element of $\fh$ with $\langle\alpha,\alpha\rangle=\mn$.\\
$M(1)$ & the vertex algebra associated to the Heisenberg algebra.\\
$\lattice$ & a non-degenerate even lattice of finite rank.\\
$\rankL$ & the rank of $\lattice$.\\
$V_{\lattice}$ & the vertex algebra associated to $\lattice$.\\
$\theta$ & the automorphism of $V_{\lattice}$ induced from the $-1$ symmetry of $\lattice$.\\
$M(1)^{+}$ & the fixed point subalgbra of $M(1)$ under the action of $\theta$.\\
$V_{\lattice}^{+}$ & the fixed point subalgbra of $V_{\lattice}$ under the action of $\theta$.\\
$\mK,\module,\mN,\mW$ & weak $M(1)^{+}$ (or $V_{\lattice}^{+}$)-modules.\\
$I(\mbox{ },x)$ & an intertwining operator for $M(1)^{+}$.\\
$\epsilon(\lu,\lv)$ & 
$\lu_{\epsilon(\lu,\lv)}\lv\neq 0\mbox{ and }\lu_{i}\lv=0\mbox{ for all }i>\epsilon(\lu,\lv)$
if $I(\lu,x)\lv\neq 0$ and $\epsilon(\lu,\lv)=-\infty$ if $I(\lu,x)\lv= 0$,
where $I : \module\times\mW\rightarrow \mN\db{x}$ is an intertwining operator and 
$\lu\in\module$, $\lv\in\mW$ \eqref{eqn:max-vanish}.\\
$\langle\omega_i\rangle X$ & the space spanned by the elements $\omega_i^{j}\lu, j\in\Z_{\geq 0}, \lu\in X$. \\
$A(V)$ & the Zhu algebra of a vertex operator algebra $V$.\\
$\omega$&$=(1/2)h(-1)^2\vac$ or $(1/2)\sum_{i=1}^{\rankL}h^{[i]}(-1)^2\vac$.\nonumber\\
$\Har$&$=(1/3)(h(-3)h(-1)\vac-h(-2)^2\vac)$.\\
$J$&$=h(-1)^4\vac-2h(-3)h(-1)\vac+(3/2)h(-2)^2\vac=-9\Har+4\omega_{-1}^2\vac-3\omega_{-3}\vac$.\\
$\ExB$&$=\ExB(\alpha)=e^{\alpha}+\theta(e^{\alpha})$ where $\alpha\in\fh$.\\
$\lE$ & an integer such that $\lE\geq \epsilon(\ExB,\lu)$ or $\lE=\epsilon(\ExB,\lu)$ for a given non-zero element $\lu$.\\
\end{tabular}


\bigskip
\noindent\textbf{Acknowledgments}\\
The author thanks the referee for pointing out a mistake in the proof of the earlier version of Proposition \ref{proposition:Zhu-Omega}.

\providecommand{\MR}{\relax\ifhmode\unskip\space\fi MR }
\providecommand{\MRhref}[2]{%
  \href{http://www.ams.org/mathscinet-getitem?mr=#1}{#2}
}
\providecommand{\href}[2]{#2}

\bibliographystyle{ijmart}

\end{document}